\documentclass[twoside,11pt]{article}

\usepackage{ifthen}
\usepackage{amssymb}
\usepackage{latexsym,amsthm}
\usepackage[final]{showkeys}
\usepackage{graphicx}
\usepackage{verbatim}
\usepackage{mathrsfs}
\usepackage{verbatim}
\usepackage{xcolor}
\usepackage{hyperref}

\makeatletter
\usepackage{amsmath}
\makeatother

\newtheorem{lemma}{Lemma}

\newtheorem{theorem}{Theorem}
\newtheorem{innercustomthm}{Theorem}
\newenvironment{customthm}[1]
  {\renewcommand\theinnercustomthm{#1}\innercustomthm}
  {\endinnercustomthm}
\newtheorem{proposition}{Proposition}
\theoremstyle{definition}
\newtheorem{remark}{Remark}

\newtheorem{example}{Example}

\newcommand{\Wnv}{W^a}
\newcommand{\R}{{\mathbb R}}
\newcommand{\N}{{\mathbb N}}
\newcommand{\LL}{\mathcal L}
\newcommand{\Ha}{\mathcal H}
\newcommand{\eps}{\varepsilon}

\newcommand{\Gauss}{\mathcal{E}_G}
\def\Xint#1{\mathchoice
       {\XXint\displaystyle\textstyle{#1}}%
       {\XXint\textstyle\scriptstyle{#1}}%
       {\XXint\scriptstyle\scriptscriptstyle{#1}}%
       {\XXint\scriptscriptstyle\scriptscriptstyle{#1}}%
       \!\int}
    \def\XXint#1#2#3{{\setbox0=\hbox{$#1{#2#3}{\int}$}
         \vcenter{\hbox{$#2#3$}}\kern-.5\wd0}}
    
\def\fint{\Xint-}
\DeclareMathOperator{\aplim}{ap\,lim}
\DeclareMathOperator{\capa}{Cap}

\DeclareMathOperator{\id}{Id}
\DeclareMathOperator{\loc}{loc}
\DeclareMathOperator{\trace}{tr}
\DeclareMathOperator{\graph}{\Gamma}
\DeclareMathOperator{\dive}{div}
\newcommand{\mrestr}{\ \rule{.4pt}{7pt}\rule{6pt}{.4pt}\ }
\newcommand{\MDir}{\mathcal{M}}
\newcommand{\MNav}{\widehat{\mathcal{M}}}

\begin{document}

\title{Minimising a relaxed Willmore functional for graphs subject to  boundary conditions}

\author{Klaus Deckelnick\thanks{e-mail:  Klaus.Deckelnick@ovgu.de}\\
Fakult\"at f\"ur Mathematik\\ Otto-von-Guericke-Universit\"at\\
Postfach 4120\\ D-39016 Magdeburg
\and
Hans-Christoph Grunau\thanks{e-mail: Hans-Christoph.Grunau@ovgu.de
}\\ Fakult\"at f\"ur Mathematik\\ Otto-von-Guericke-Universit\"at\\
Postfach 4120\\ D-39016 Magdeburg
\and 
Matthias R\"oger\thanks{e-mail:  matthias.roeger@tu-dortmund.de}\\
Fakult\"at f\"ur Mathematik\\
Technische Universit\"at Dortmund\\
Vogelpothsweg 87\\
D-44227 Dortmund
}
\maketitle

\begin{abstract} 
For a bounded  smooth domain in the plane and smooth boundary data we consider the minimisation 
of the Willmore functional for graphs subject to Dirichlet or Navier boundary conditions. 
For $H^2$-regular graphs we show that bounds for the Willmore energy imply area and diameter bounds. 
We then consider the $L^1$-lower semicontinuous relaxation of the Willmore functional, 
which is shown to be indeed its largest possible extension,
and characterise properties  of functions with finite relaxed energy. 
In particular, we deduce compactness and lower-bound estimates for energy-bounded sequences. 
The lower bound is given by a functional that describes the contribution by the regular part of the graph and is defined for a suitable subset of $BV(\Omega)$. 
We further show that finite relaxed Willmore energy implies the attainment of the Dirichlet boundary data in an appropriate sense, and obtain the 
existence of a minimiser in $L^\infty\cap BV$ for the relaxed energy. Finally, we extend our results to Navier boundary conditions and 
more general curvature energies of Canham--Helfrich type.

\end{abstract}

\section{Introduction and main results}
\label{sec:intro}

The present paper is intended as an analogue for the Willmore functional of the BV-approach
of minimising the non-parametric area functional under Dirichlet boundary conditions  (see \cite[Theorem 14.5]{Giusti}). 
We therefore consider for two-dimensional graphs $\Gamma\subset\R^3$ the following combination of Willmore 
functional (cf. \cite{Willmore})\footnote{This functional had indeed shown up already at the
beginning of the 19th century. For historical and mathematical background information  on the Willmore functional
 one may see \cite{MarquesNeves,Nitsche}.} and total Gau\ss\ curvature
\begin{align*}
	\mathcal W_\gamma(\Gamma) \,=\, \frac{1}{4} \int_{\Gamma} H^2 \,dS	- \gamma \int_{\Gamma} K\,dS,
\end{align*}
where $\gamma\in \R$ is a constant\footnote{For simplicity we call 
$W_\gamma$ Willmore functional also for $\gamma\neq 0$.},
$H$ and $K$ denote the mean and Gau\ss\ curvature and are defined as the sum respectively the product of the principal 
curvatures.
We investigate how and to what extent a direct method of the calculus of variations can be applied to the respective minimisation problem, 
subject to boundary conditions. We therefore need to identify a suitable class of functions and a suitable generalisation of the Willmore energy
 that allow for compactness and lower semicontinuity properties.

\subsection{The Willmore functional and boundary value problems for the Willmore equation}
\label{sec:intro_1}

Let $\Omega \subset \mathbb{R}^2$ be a bounded domain with a $C^2$--boundary
and exterior unit normal field $\nu$, let $\varphi:\overline{\Omega}\to \mathbb{R}$ be a sufficiently smooth (at least $\varphi\in C^2(\overline{\Omega}) $) 
boundary datum, and fix a parameter $\gamma\in \mathbb{R}$.
Our aim is to minimise the Willmore functional $\mathcal W_\gamma$ in the class of graphs 
$$
\graph(u)= \lbrace (x,u(x)) \, | \, x \in \Omega \rbrace
$$
of suitable functions
$
u:\overline{\Omega}\to \mathbb{R},
$
and subject to a boundary condition. We therefore consider 
\begin{align}
	W_\gamma(u) :=\mathcal W_\gamma(\Gamma(u)) := \frac{1}{4} \int_{\Omega} H^2 \; \sqrt{1+ | \nabla u |^2} \, dx 
	- \gamma \int_{\Omega} K\; \sqrt{1+ | \nabla u |^2}\, dx \label{eq:def-Wgamma}
\end{align}
either in the classes
\begin{equation}\label{eq:def_class_dirichlet}
 \{u:\overline{\Omega}\to \mathbb{R} :\, 
u=\varphi, \quad \frac{\partial u}{\partial \nu} = \frac{\partial \varphi}{\partial \nu} \mbox{ on } \partial \Omega \}
\end{equation}
of clamped graphs or
\begin{equation}\label{eq:def_class_navier}
  \{u:\overline{\Omega}\to \mathbb{R} :\,  u=\varphi\mbox{ on } \partial \Omega\}
\end{equation}
of hinged graphs, respectively.
According to \cite{Nitsche}, $\gamma\in [0,1] $ is a physically relevant condition,
which implies that $\frac{1}{4}H^2- \gamma K\ge 0$. We expect that this condition%
--among others--will be needed to ensure regularity of a minimiser of $W_\gamma$.
For the compactness and lower semicontinuity properties stated in the present paper, however, we allow for arbitrary $\gamma\in\R$.

In order to explain  the notion of Dirichlet and Navier boundary value problems for Willmore surfaces
let us assume that we have a smooth minimiser of  $W_\gamma$ in the class (\ref{eq:def_class_dirichlet})
or (\ref{eq:def_class_navier}), respectively.
In the first case, i.e. considering a minimiser in the class of clamped graphs, one would have a solution  for the Willmore equation 
\begin{equation}\label{eq:Willmore}
 \Delta_{\graph(u)} H +2H\Big(\frac{1}{4}H^2-K\Big)=0\qquad \mbox{\ in\ }\Omega
\end{equation}
under Dirichlet boundary conditions
\begin{equation} \label{eq:dirichlet}
\displaystyle 
u=\varphi, \quad \frac{\partial u}{\partial \nu} = \frac{\partial \varphi}{\partial \nu} \qquad\mbox{ on } \partial \Omega,
\end{equation}
see \cite[(25)]{Nitsche}.
Here $\Delta_{\graph(u)}$ denotes the Laplace-Beltrami operator on ${\graph(u)}$ with respect to the first 
fundamental form. According to Remarks~\ref{rem:geod_curv} and \ref{rem:Gauss_Bonnet} below the shape of $\Omega$ and the 
Dirichlet data (\ref{eq:dirichlet}) completely determine $\int_{\Omega} K\; \sqrt{1+ | \nabla u |^2}\, dx$.
So, in order to solve the Dirichlet  problem, the parameter $\gamma$ does not play any role and without loss 
of generality we may restrict ourselves to minimising $W_0$.

Let us assume now that $u$ is a smooth minimiser  for $W_\gamma$ in the class (\ref{eq:def_class_navier})
of hinged graphs. Such a minimiser then  solves the  Willmore equation 
\begin{equation}\tag{\ref{eq:Willmore}}
 \Delta_{\graph(u)} H +2H\Big(\frac{1}{4}H^2-K\Big)=0\qquad \mbox{\ in\ }\Omega
\end{equation}
under Navier boundary conditions
\begin{equation} \label{eq:navier}
\displaystyle 
u=\varphi, \quad H= 2\gamma \kappa_N\qquad \mbox{ on } \partial \Omega.
\end{equation}
Here,  $\kappa_N$ denotes the normal curvature of the boundary curve $\partial \graph(u)$ with respect to the upward
pointing unit normal vector field
$N:=\frac{1}{\sqrt{1+| \nabla u|^2}}(-\nabla u,1)$
of $\graph(u)$. 
The second boundary condition $H= 2\gamma \kappa_N $ arises as a natural one  due to 
the larger class of admissible comparison functions, see  \cite[(32)]{Nitsche}.
The case $\gamma=0$, i.e.~prescribing $H|_{\partial \Omega}=0$, is special since here 
one may just seek solutions of the minimal surface equation subject to the boundary condition $u |_{\partial \Omega}=
\varphi|_{\partial \Omega}$, see e.g. \cite[Sect. 14]{GilbargTrudinger} or \cite{Giusti}. A recent paper by Bergner and Jakob \cite{BergnerJakob}
ensures that one even does not miss solutions when using this approach. 

These observations motivate speaking of Dirichlet and Navier boundary conditions in Sections~\ref{sec:minimiser_Dirichlet}
and \ref{sec:minimiser_Navier} respectively, although we do in general not expect sufficient regularity of the solutions 
constructed in Theorem~\ref{thm:existence_minimiser} and Remark~\ref{rem:Wgamma} 
to solve the above mentioned boundary value problems in a classical sense. 

Sch\"atzle \cite{Schaetzle} solved the Dirichlet problem for Willmore surfaces in the very general context
of immersions in $\mathbb{S}^3$.
This approach, however, does not give easy access to more detailed geometric information on the solution.
 In particular, even in the case of rather simple and regular boundary data it is not obvious how to single out graph solutions under suitable assumptions on the data.
 Concerning classical solvability of boundary value problems for the 
Willmore equation under \emph{symmetry} assumptions one may see \cite{BDAF,DADW,DFGS,DeckelnickGrunau} and references therein.
According to \cite{Dall'Acqua}, for strictly star--shaped $\Omega$ and $\varphi=0$, the constant function $u= 0$ is the unique
solution to the Dirichlet problem. Due to the strongly nonlinear character 
and the lack of convexity of this problem we do in general not expect uniqueness;
numerical evidence is given in \cite{DKS}.

We remark that many papers have dealt with closed Willmore surfaces (compact without boundary); we mention only 
\cite{BauerKuwert,Simon} for existence of (minimising) Willmore surfaces of any prescribed genus. Further information 
can also be found in the lecture notes \cite{KuSc12} and the survey article \cite{MarquesNeves} on the recent proof of the Willmore conjecture.

\subsection{Main results}
\label{sec:intro_2}

In our work the major benefit of working with graphs, i.e.~using  a  non-parametric approach, is the validity of a-priori diameter and area bounds which
are not available in the general parametric setting. More precisely, the corresponding 
result (see Section~\ref{sec:bounds} below) reads as follows:

\begin{customthm}{\ref{thm:aprioribounds}}
Suppose that $u \in H^2(\Omega)$ satisfies $u-\varphi \in H^1_0(\Omega)$. Then there exists a constant $C$ that only depends
on $\Omega$  and $ \Vert \varphi \Vert_{W^{2,1}(\partial \Omega)}  $ such that
\begin{equation*} 
\displaystyle 
\sup_{x \in \Omega} | u(x)| + \int_{\Omega} \sqrt{1+|\nabla u(x) |^2} \, dx \leq C \bigl( W_0(u)^2 +1 \bigr).
\end{equation*}
\end{customthm}
We will also present several examples that in particular demonstrate that no a-priori bounds in $W^{1,p}(\Omega)$ 
in terms of the Willmore energy are available for any $1<p\le \infty$. 
Unlike the axially symmetric setting (see e.g. \cite{DFGS}) we have further not  yet  succeeded to modify minimising sequences such that they
obey stronger bounds  than in Theorem~\ref{thm:aprioribounds}.

Our main results   are stated in Theorem~\ref{thm:L1lowersemicontinuity}. 
We show that sequences $(u_k)_{k\in\N}\subset H^2(\Omega)$ with uniformly bounded Willmore energy 
and obeying the boundary condition $(u_k-\varphi)\in H^1_0(\Omega)$
have $L^1(\Omega)$-convergent subsequences. 
Limit points belong to $BV(\Omega)\cap L^\infty(\Omega)$ and enjoy additional (weak) regularity properties that allow for the 
definition of an absolutely continuous contribution to the Willmore functional (see Section~\ref{sec:lowersemicontinuous} for details). 
This contribution then gives a lower bound for the energies of the approximating sequence. For simplicity we state here only a corollary of 
Theorem~\ref{thm:L1lowersemicontinuity} and assume that the limit point $u$ belongs to $W^{1,1}(\Omega)$, which allows to control the full Willmore functional.

\begin{customthm}{\ref{thm:L1lowersemicontinuity}'}
Let $(u_k)_{k\in\N}$ be a given sequence in $H^2(\Omega)$ that satisfies
\begin{align*}
	u_k-\varphi \in H^1_0(\Omega)\text{ for all }k\in\N\quad\text{ and }\quad\liminf_{k\to\infty} W_0(u_k) \,<\,\infty.
\end{align*}
Then there exists a subsequence $k\to\infty$ and $u\in BV(\Omega)\cap L^\infty(\Omega)$ with
\begin{equation*}
 u_k\to u \text{ in } L^1(\Omega)\quad (k\to\infty).
\end{equation*}
If in addition $u\in W^{1,1}(\Omega)$ then the mean curvature $H=\nabla\cdot \frac{\nabla u}{\sqrt{1+|\nabla u|^2}}\in L^2(\Omega)$ exists in the weak sense and  
\begin{equation*}
 \frac{1}{4}\int_\Omega H^2  \sqrt{1+|\nabla u|^2}\,dx 
 \,\le\, \liminf_{k\to\infty} W_0(u_k).
\end{equation*}
holds.
\end{customthm}

In order to simplify the presentation, in the remainder of this introduction we restrict ourselves to Dirichlet boundary conditions (\ref{eq:dirichlet}) 
and to the case $\gamma=0$. 
As explained before, for minimising sequences, or more generally for sequences with uniformly 
bounded Willmore functional, we do not have stronger uniform bounds than those in Theorem~\ref{thm:aprioribounds}.
So,  the regularity of limit points can at a first instance  not be proved 
to be  better than $L^\infty (\Omega) \cap BV(\Omega)$. On this space, however, the
Willmore functional is not defined in the classical sense and we therefore introduce the $L^1$-lower semicontinuous relaxation of the Willmore 
functional:
$$
\overline{W}:L^1( \Omega)\to [0,\infty],\quad \overline{W}(u):=\inf \{ \liminf_{k\to\infty} W_0(u_k): \MDir\ni u_k \to u 
\mbox{\ in\ } L^1 (\Omega)\},
$$
where
\begin{align}
	\MDir\,:=\, \{v\in H^2(\Omega):v-\varphi \in H^2_0(\Omega)\}. \label{eq:def-MDir}
\end{align}
For geometric curvature functionals such relaxations are well established,
see e.g. \cite{AmbrosioMasnou,BellettiniDalMasoPaolini,BellettiniMugnai} and the references therein. One advantage is that 
lower-semicontinuity properties are immediately obtained; on the other hand, a more explicit characterisation of the relaxation is often difficult. 
However, we prove in Theorem~\ref{thm:coinidence_relaxation} that $W_0$ and $ \overline{W}$ coincide in $\MDir$, so that $\overline{W}$ is actually an extension of $W_0$.

As a corollary of Theorem~\ref{thm:L1lowersemicontinuity} we are able to prove existence of a minimiser for the extended functional $ \overline{W}$:
\begin{customthm}{\ref{thm:existence_minimiser}}
There exists a function $u\in BV( \Omega)\cap L^\infty (\Omega)$ such that
$$
\forall v\in L^1( \Omega):\quad \overline{W}(u)\le \overline{W}(v).
$$
\end{customthm}

The regularity properties stated in Theorem~\ref{thm:L1lowersemicontinuity}  are in particular satisfied for any function 
$u\in L^1(\Omega)$ with $\overline{W}(u)<\infty$ and so, for the minimiser constructed above.
Furthermore, in Proposition~\ref{prop:W_vs_overline_W} we prove that $ \overline{W}(u)<\infty$ not only allows for defining a 
generalised Willmore functional (or rather the absolutely continuous part), but also encodes attainment of the boundary conditions (\ref{eq:dirichlet}).

\medskip

The proofs of these results all heavily rely on the area and diameter bounds provided by Theorem~\ref{thm:aprioribounds}.
Together with the boundedness of the Willmore energies this yields sufficiently strong compactness properties for several 
$H^1$-bounded  auxiliary 
sequences such as $q_k=(1+|\nabla u_k|^2)^{-5/4}$ and $v_k=q_k\nabla u_k $. In particular we are able to deduce that for limits 
$u\in BV(\Omega)\cap L^\infty(\Omega)$ as in Theorem~\ref{thm:L1lowersemicontinuity} that $v=q\nabla u$ holds as 
vector-valued Radon measures, for $v,q\in H^1(\Omega)$. Moreover, the set $\{q=0\}$ describes the set where the graph of $u$ may become vertical. 
Our results then are deduced by exploiting several fine properties of Sobolev and $BV$ functions. 

\medskip
Restricting ourselves to the graph case --i.e.~working in the non-parametric framework-- 
allows to use relatively elementary tools (compared to the use of geometric measure theory methods in the parametric case),
 but on the other hand introduces additional difficulties that are due to the particular choice of parametrisation and are in 
particular related to the possible occurrence of vertical parts of the graph when passing to a limit.
The condition of being a graph imposes an obstruction 
to the class of admissible ``surfaces''. Minimising the Willmore functional in this class means solving a kind 
of an obstacle problem, as long as one cannot prove $C^1$-estimates or $C^1$-regularity: 
We expect our minimiser to solve the Willmore equation on
the non-vertical parts of the graph while this can in general not be expected on the
vertical parts.  Functions having vertical parts lie somehow on the ``boundary'' of the set
of admissible functions since some variations would result in surfaces which are no
longer graphs but could possibly nevertheless have smaller Willmore energy.
We think that extra conditions on the data $\Omega$, $\varphi$, and $\gamma$ will be needed
to prove that our minimiser  is  indeed smooth and attains the
boundary conditions in a classical sense,  but such a characterisation is out of the scope of the present paper.

The paper is organised as follows. In the next section we first state some definitions from differential geometry and properties of Sobolev and 
$BV$ functions, and then prove some basic estimates. 
Section~\ref{sec:bounds} presents the main a-priori bounds and examples that show that these bounds are in some sense optimal. 
The main compactness and lower-semicontinuity properties are formulated and proved in Section~\ref{sec:lowersemicontinuous}. 
The last section derives the implications for the minimisation of the (relaxed) Willmore functional and in particular discusses in which sense the boundary 
conditions are attained for functions with finite relaxed energy. Finally, extensions to 
more general functionals of Canham--Helfrich type are indicated.
%
%
\section{Preliminaries and basic estimates}
\label{sec:useful}
\subsection{Differential geometry of graphs}
For a smooth function
$u: \Omega \rightarrow \mathbb{R}$ we let
\begin{displaymath}
\graph(u):= \lbrace (x,u(x)) \, | \, x \in \Omega \rbrace
\end{displaymath}
be its graph with unit normal field $N:=\frac{1}{\sqrt{1+| \nabla u|^2}}(-\nabla u,1)$.  The first and second 
fundamental forms of $\graph(u)$ are given by
\begin{displaymath}
(g_{ij}) = 
\begin{pmatrix} 
1+ u_{x^1}^2 & u_{x^1} u_{x^2} \\
u_{x^1} u_{x^2} & 1 + u_{x^2}^2
\end{pmatrix},
\qquad
A=(h_{ij}) =  \frac{1}{Q} (u_{x^i x^j}),
\end{displaymath}
where $Q= \sqrt{1+ | \nabla u|^2}$ denotes the area element. The mean curvature and the Gau{\ss} curvature  of
$\graph(u)$ are then given by
\begin{align}
	H &=  \nabla \cdot \frac{\nabla u}{Q}\,=\, \frac{1}{Q}\Bigl(\id -w\otimes w\Bigr): D^2u, \label{eq:Hdef}\\
	K &=  \frac{\operatorname{det}D^2 u}{Q^4} \,=\, \det Dw, \label{eq:Kdef}
\end{align}
where we have set $w:=\nabla u / Q$.
In particular, the Willmore functional for the graph of $u$ reads
\begin{eqnarray} 
W_\gamma (u) & = & \frac{1}{4} \int_{\Omega} H^2 \; \sqrt{1+ | \nabla u |^2} \, dx -\gamma \int_{\Omega} K \; \sqrt{1+ | \nabla u |^2} \, dx\nonumber\\
&=& \frac{1}{4} \int_{\Omega} | \nabla \cdot \Bigl( \frac{ \nabla u}{Q} \Bigr) |^2 \; Q \, dx-\gamma \int_{\Omega} \frac{\operatorname{det}D^2 u}{Q^3} \, dx.
\label{eq:Darstellung:Willmore}
\end{eqnarray}
Mean curvature, Gau{\ss} curvature and the length of the second fundamental form  
\begin{displaymath}
| A |_g^2 = \sum^2_{i,j,k,\ell=1} g^{ij} g^{k\ell} h_{ik} h_{j\ell}\,=\, \trace (g^{-1}Ag^{-1}A), \qquad (g^{ij})=(g_{ij})^{-1},
\end{displaymath}
are related by the formula
\begin{equation} \label{eq:AH}
\displaystyle
|A|_g^2 = H^2 - 2K.
\end{equation}
\subsection{Functions of bounded variation and fine properties of Sobolev functions}
\label{sec:prelim}
We denote by $B_r(x)$ for $x\in\R^n$, $r>0$ the corresponding open ball, by $\LL^n$ the $n$-dimensional Lebesgue measure and by $\Ha^k$ the 
$k$-dimensional Hausdorff measure. We set $|A|=\LL^n(A)$ for $A\subset\R^n$. 
The precise representative of a function $u\in L^1_{\loc}(\R^n)$,
\begin{align*}
	u^*(x) \,:=\, \lim_{r\downarrow 0}\fint_{B_r(x)} u(y)\,dy,
\end{align*}
is well-defined almost everywhere, where we have used the notation $\fint_{B_r(x)} u:= \frac{1}{|B_r(x)|}\int_{B_r(x)}u$. 
The Lebesgue points of $u$ are given by all $x\in\R^n$ such that
\begin{align*}
	\lim_{r\downarrow 0}\fint_{B_r(x)} |u(x)-u(y)|\,dy\,=\, 0.
\end{align*}
The usual Sobolev spaces are denoted by $H^\ell(\Omega)$, $\ell\in \N_0$, $H(\operatorname{div},\Omega)$ denotes  the space of 
$L^2(\Omega,\mathbb{R}^n)$--vector fields which have a weak 
divergence in $L^2(\Omega)$.
For the definition and properties of the space $H(\operatorname{div},\Omega)$ see e.g. 
\cite[Chapter 1, Section 1.2]{Temam}.

We next recall some basic definitions and properties of functions of bounded variation. For a detailed exposition we refer to the book of Ambrosio, 
Fusco and Pallara \cite{AmFP00}.

A function $u\in L^1(\Omega)$ belongs to the space of functions of bounded variation if the distributional derivatives $D_i u$ are given by 
finite Radon measures on $\Omega$. We then write $u\in BV(\Omega)$ and denote by $\nabla u$ the vector-valued Radon measure with 
components $D_i u$. For $u\in BV(\Omega)$ the total variation of  $\nabla u$ is given by
\begin{align*}
	\int_\Omega |\nabla u| \,=\, \sup \bigl\{ \int_\Omega u \nabla\cdot\varphi\,dx \,:\, \varphi\in C^1_0(\Omega,\R^n), \|\varphi\|_\infty\leq 1\bigr\}.
\end{align*}
The measure $\nabla u$ can be decomposed as
\begin{align}
	\nabla u \,=\, \nabla^a u\LL^n + \nabla^s u\,=\, \nabla^a u\LL^n + \nabla^ju + \nabla^c u,
\end{align}
where $\nabla^a u\LL^n$ denotes the absolutely continuous part of $\nabla u$ with respect to $\LL^n$, and $\nabla^s u$, $\nabla^j u$, $\nabla^c u$ 
are the singular part, the jump part and the Cantor part of $\nabla u$, respectively. Letting
\begin{align*}
	\Sigma_u \,&:=\, \{x\in\Omega \,:\, \lim_{\varrho\downarrow 0}\varrho^{-n}|\nabla u|(B_\varrho(x)) =\infty\},\quad
	\Theta_u \,:=\, \{x\in\Omega \,:\, \liminf_{\varrho\downarrow 0}\varrho^{1-n}|\nabla u|(B_\varrho(x)) >0\}
\end{align*}
we have $\nabla^a u \LL^n= \nabla u \mrestr (\Omega\setminus \Sigma_u)$, $\nabla^j u = \nabla u\mrestr\Theta_u$, 
$\nabla^c u = \nabla u \mrestr(\Sigma_u\setminus \Theta_u)$, see \cite[Proposition 3.92]{AmFP00}. 
The set $\Sigma_u$ has Lebesgue measure zero, see \cite[Theorem 1.6.1]{EvGa92}.
Moreover, by \cite[Theorem 3.78]{AmFP00} 
$\nabla^j u = (u^+-u^-)\otimes\nu_u\Ha^{n-1}\mrestr J_u$, where the approximate jump set $J_u\subset \Theta_u$ (see \cite[Definition 3.67]{AmFP00}) is 
$(n-1)$-rectifiable, $\nu_u$ is a Borel unit normal vector field to $J_u$ and $u^+,u^-$ are the traces of $u$ on $J_u$.
The complement $S_u$ of the set of Lebesgue points of $u$ is a Borel set with $\LL^n$-measure zero and satisfies $\Ha^{n-1}(S_u\setminus J_u)=0$, 
see \cite[Definition 3.63 and Theorem 3.78]{AmFP00}.

For a function $u\in BV(\Omega)$ we call $x\in\partial\Omega$ a Lebesgue boundary point if
\begin{align*}
	\lim_{r\downarrow 0}\fint_{B_r(x)\cap\Omega} |u(x)-u(y)|\,dy\,=\, 0,
\end{align*}
where $\fint_{B_r(x)\cap\Omega} u:= \frac{1}{|B_r(x)\cap\Omega|}\int_{B_r(x)\cap\Omega}u$
and $u(x)$ is defined in the sense of boundary traces.

\medskip

We next recall the notion of capacity and some fine properties of Sobolev functions. We follow \cite{EvGa92}. 
For $A\subset\R^n$ and $1\le p<n$ the $p$-capacity is defined as
\begin{align*}
	\capa_p(A) \,:=\, \inf \bigl\{\int_{\R^n} |\nabla f|^p\,dx \,:\, f\geq 1 \text{ in a neighbourhood of }A,\,f\geq 0\bigr\},
\end{align*}
where the infimum is taken over all $f\in L^{p^*}(\R^n)$ with $\nabla f\in L^p (\R^n,\R^n)$, $p^*=\frac{np}{n-p}$.\\
If $\capa_p(A)=0$ then $\Ha^s(A)=0$ for all $s>n-p$ \cite[Theorem 4.7.4]{EvGa92}.

For a function $u\in W^{1,p}(\R^n)$,  $1\leq p<n$, there exists a Borel set $E$ of $p$-capacity zero such that 
the precise representative $u^*$ of $u$ is  well defined on $\R^n\setminus E$ and 
each $x\in \mathbb{R}^n\setminus E$ is a Lebesgue point of $u^*$ 
\cite[Theorem 4.8.1]{EvGa92}. Moreover, for every $\eps>0$ there exists an open set $V$ with $\capa_p(V)\leq \eps$ such 
that $u^*$ is continuous on $\R^n\setminus V$.

\medskip

We say that $f:\R^n\to\R^m$ has an approximate limit at $x\in\R^n$ if there exists $a\in\R^m$ such that for every $\eps>0$
\begin{align*}
	\lim_{r\downarrow 0} \frac{\left| B_r(x)\cap\{|f-a|\geq\eps\}\right| }{|B_r(x)|} \,=\, 0,
\end{align*}
see \cite[Section 1.7.2]{EvGa92}. In this case the approximate limit $\aplim_{y\to x}f(y):=a$ is uniquely determined. 
We call $f$ approximately continuous at $x\in\R^n$ if $\aplim_{y\to x}f(y)=f(x)$.
By \cite[Remark 5.9.2]{Ziem89} $f:\R^n\to\R^m$
 is approximately continuous at $x\in\R^n$ if and only if there exists a measurable set $E\subset\R^n$ with $x\in E$ such that $f|_{E}$ is continuous at $x$ 
and the set $E$ has full density in $x$, that is
\begin{align*}
	\lim_{r\downarrow 0}\frac{\left| B_r(x)\cap E \right|}{|B_r(x)|} \,=\, 1.
\end{align*}
Therefore the products (quotients) of approximately continuous real functions are approximately continuous (in all points where the denominator does not vanish).

We say that $f:\Omega\to\R^m$ has an approximate limit at $x\in\partial\Omega$ if there exists $a\in\R^m$ such that for every $\eps>0$
\begin{align*}
	\lim_{r\downarrow 0} \frac{\left| B_r(x)\cap\Omega\cap\{|f-a|\geq\eps\}\right| }{|B_r(x)\cap\Omega|} \,=\, 0.
\end{align*}

\subsection{Basic estimates}\label{sec:basic-estimates}
The following result shows how the second derivatives of $u$ are controlled in terms of 
$|A|_g^2$.

\begin{lemma}
Let $|D^2 u|^2=u_{x^1x^1}^2+2u_{x^1x^2}^2+u_{x^2x^2}^2$ denote the euclidean norm
of the Hessian of $u$. Then
\begin{equation}\label{eq:Hessian_vs_secondff}
\frac{1}{Q(x)^2} | D^2 u(x) |^2\ge 
 |A(x) |_g^2\ge \frac{1}{Q(x)^6} | D^2 u(x) |^2.
\end{equation}
\end{lemma}

This possibly very strong deviation of $|A(x) |_g^2$ from $| D^2 u(x) |^2$ is one
of the main difficulties in deducing a-priori estimates for minimising sequences 
of the Willmore functional.

\begin{proof}
 We have that 
$$
\left( g^{ij}(x) \right)_{ij} = \frac{1}{Q^2}
\begin{pmatrix}
 1+u^2_{x^2} &- u_{x^1} u _{x^2} \\
- u_{x^1} u _{x^2} &  1+u^2_{x^1}
\end{pmatrix} \,=\, \frac{1}{Q^2} \Bigl(\id + \nabla u^\perp\otimes\nabla u^\perp\Bigr)
$$
is a symmetric positive definite matrix with smallest eigenvalue equal to $\frac{1}{Q^2}$.
One the one hand this yields the estimate:
$$
\forall \eta \in \mathbb{R}^2:\qquad |\eta|^2\ge  \sum_{i,j} g^{ij}(x) \eta_i\eta_j\ge \frac{1}{Q^2} |\eta|^2.
$$
On the other hand we find a uniquely determined symmetric positive definite square root $\left(b^{ij}(x)\right)$ 
of $\left( g^{ij}(x) \right)$, i.e.
$$
 g^{ij}(x) =\sum_{k,\ell} b^{ik}(x)\delta_{k \ell}  b^{\ell j}(x).
$$
Denoting
$
v^i_j:=\sum_\ell b^{i\ell}h_{\ell j}
$
we see that
\begin{eqnarray}
| A |_g^2 &=& \sum_{i,j,k,\ell} g^{ij} g^{k\ell} h_{ik} h_{j\ell}
=\sum_{i,j,k,\ell,m,n} g^{k\ell} b^{im}\delta_{mn}b^{nj} h_{ik} h_{j\ell}  \nonumber\\
&=&\sum_{k,\ell,m,n}g^{k\ell} v^m_k   \delta_{mn} v^n_\ell
=\sum_{k,\ell,m} g^{k\ell} v^m_k v^m_\ell \nonumber \\
&\ge &  \frac{1}{Q^2}\sum_{k,m} \left( v^m_k \right)^2
= \frac{1}{Q^2}\sum_{i,j,k,\ell}\delta_{ij} \delta^{k\ell}v^i_k v^j_\ell
=\frac{1}{Q^2}\sum_{i,j,k,\ell,m,n}\delta_{ij} \delta^{k\ell}b^{i m}h_{m k}b^{j n}h_{n \ell}  \nonumber \\
&=&\frac{1}{Q^2}\sum_{k,m,n}g^{mn}  h_{m k} h_{n k} 
\ge  \frac{1}{Q^4}\sum_{m,k} \left(h_{m k}\right)^2
= \frac{1}{Q^6}\sum_{m,k}\left(u_{x^m x^k}\right)^2=\frac{1}{Q^6}|D^2 u|^2. \label{est}
\end{eqnarray}
As for the bound from above we find by using similar calculations as before that
$$
| A |_g^2 \le \sum_{m,k} \left(h_{m k}\right)^2
= \frac{1}{Q^2}\sum_{m,k}\left(u_{x^m x^k}\right)^2=\frac{1}{Q^2}|D^2 u|^2.
$$
\end{proof}

In what follows the geodesic 
curvature of the boundary curve $\partial\graph(u)$
with respect to the surface $\graph(u)$ will be of some importance. We derive here an explicit estimate and representation that are used below.

\begin{remark}\label{rem:geod_curv} 
We consider $u\in H^2(\Omega)$ satisfying $(u-\varphi)\in H^1_0 (\Omega)$.
Let $s\mapsto Y(s)\in \partial\graph(u)$
denote a positively oriented parametrisation of (a connected component of)
the boundary $\partial\graph(u)$. Positive orientation means that at any point  $p\in\partial \graph(u)$,
 the determinant of the unit tangent vector, the unit co-normal 
pointing inward to  $\graph(u)$ and $N(p)$ is positive.

Then its (signed) geodesic curvature is given by
$$
\kappa_g (s)=\frac{1}{|Y'(s)|^3}
\operatorname{det} \left( Y'(s), Y''(s),N(Y(s))\right).
$$
We take now a positively oriented parametrisation $s\mapsto  c (s) \in \partial\Omega$ of
(a connected component of) $\partial\Omega$ with respect to its arclength
so that with the natural unit tangent vector $\tau(s)= c '(s)$, we have that $(\nu(  c (s) ),\tau(s))$
form a positively oriented orthonormal basis of  $\mathbb{R}^2$. In particular we have
that $\nu^1=\tau^2$, $\nu^2=-\tau^1$ and $\tau'(s)=-\kappa (s)\nu(  c (s) )$ with $\kappa$
the (signed) curvature of $\partial\Omega$ (being nonnegative on the ``convex'' parts of $\partial\Omega$). 
%
With a slight abuse of notation we write
$$
\varphi(s)=\varphi ( c (s) ),\qquad u_\nu(s)=\frac{\partial u}{\partial \nu}( c (s) ),\qquad  Y(s)=(c (s),\varphi(s))^T
$$
and find by using $u=\varphi$ on $\partial \Omega$ that
$$
N(Y(s))=\frac{1}{\sqrt{1+\varphi'(s)^2+u_\nu(s)^2}}
\begin{pmatrix}
 -(\tau^1 \varphi' + \tau^2 u_\nu)\\
-(\tau^2 \varphi' -\tau^1 u_\nu) \\
1
\end{pmatrix}.
$$
For the geodesic 
curvature of $ \partial\graph(u)$ we obtain, using $\tau^1\,(\tau^1)'+\tau^2\, (\tau^2)'=0$,
\begin{eqnarray*}
 \kappa_g (s)&=&
\frac{-u_\nu(s) \varphi''(s)+\kappa (s) (1+\varphi'(s)^2)}{(1+\varphi'(s)^2+u_\nu (s)^2)^{1/2} (1+\varphi'(s)^2 )^{3/2}}.
\end{eqnarray*}
This formula shows in particular that the geodesic curvature of $ \partial\graph(u)$ 
as a curve in the unknown surface $\graph(u)$ 
can be computed just from its Dirichlet data (\ref{eq:dirichlet}).
We observe that the assumption  $u=\varphi$ on $\partial \Omega$ already allows for estimating
\begin{displaymath}
| \kappa_g (s)  |  \leq  \frac{| \varphi''(s) | +   | \kappa(s) | }{(1+\varphi'(s)^2 )^{1/2}},
\end{displaymath}
hence
\begin{equation}\label{eq:ged_curv}
 \int_{\partial \graph(u)} |\kappa_g (s) | \, ds  \leq \int_{\partial\Omega} ( | \varphi''(s) | + | \kappa(s) | ) \, ds.
\end{equation}

%
\end{remark}

\begin{remark}\label{rem:Gauss_Bonnet}
By virtue of the Gau\ss-Bonnet formula
$$
  \int_{\Omega} K Q \, dx +\int_{\partial \graph(u)} \kappa_g ds=2 \pi \chi(\graph(u))=2 \pi \chi(\Omega),
$$
the integral over the Gau\ss\ curvature is given by the boundary integral of the geodesic curvature and the 
 Euler characteristic $\chi(\Omega)$ of the smoothly bounded domain $\Omega\subset\mathbb{R}^2$.
The Euler characteristic is defined as usual by means of triangulations. 
If $\Omega$ is $m$-fold connected, i.e. $\partial\Omega$ consists of $m$ connected components 
($\Omega$ contains $(m-1)$ holes), then $\chi(\Omega)=2-m$. See formula (13) on p. 38 in  \cite{DHS}.
 
In particular, the total Gau\ss\ curvature $\int_{\Omega} K Q \, dx$ of a function $u\in C^2(\overline{\Omega})$ 
is already determined by the Dirichlet boundary condition \eqref{eq:dirichlet}.
\end{remark}

\vspace{2mm}

In Theorem~\ref{thm:aprioribounds} below we shall deduce maximum modulus and area
estimates in terms of integral norms of the second fundamental form. In the following lemma 
we show first how to bound these by the Willmore functional, the data and the Euler
characteristic $\chi(\Omega)$ of  the smoothly bounded domain $\Omega\subset\mathbb{R}^2$.

\begin{lemma}\label{lemma:Abounds}
Suppose that $u \in H^2(\Omega)$ satisfies $u-\varphi \in H^1_0(\Omega)$. Then
\begin{align} 
	\Big| \int_\Omega KQ\,dx\Big| \,&\leq\,  \| \varphi \|_{W^{2,1}(\partial\Omega)}+ \| \kappa \|_{L^1(\partial\Omega)} + 2 \pi |\chi(\Omega)|, \label{eq:Kbound}\\
	\label{eq:Abound}
	\int_{\Omega} | A |_g^2 Q \, dx  &\leq  4 W_0(u) + 2 \bigl( \| \varphi \|_{W^{2,1}(\partial\Omega)}+ \| \kappa \|_{L^1(\partial\Omega)} \bigr) - 4 \pi \chi(\Omega),
\end{align}
where $\| \varphi \|_{W^{2,1}(\partial\Omega)}= \| \varphi \circ c \|_{W^{2,1}(I)}$ and 
$c:I \rightarrow \mathbb{R}^2$ is an arclength--parametrisation of $\partial \Omega$.
Moreover, the functionals $W_0$ and $W_\gamma$ are closely related:
\begin{equation}\label{eq:rel_W0_Wgamma} 
 \left|  W_0(u) -W_\gamma (u) \right| \le | \gamma  | \cdot\left(  \| \varphi \|_{W^{2,1}(\partial\Omega)}+ \| \kappa \|_{L^1(\partial\Omega)}+ 2\pi \left| \chi(\Omega) \right| \right).
\end{equation}

\end{lemma}

\begin{proof} Let us first assume that $u \in C^2(\overline{\Omega})$ and $u=\varphi$ on $\partial \Omega$.
We use the notation and same orientation as in Remark~\ref{rem:geod_curv}. According to \eqref{eq:ged_curv}
\begin{equation*}
 \int_{\partial \graph(u)} | \kappa_g (s) | ds  \leq \int_{\partial\Omega} ( | \varphi''(s) | + | \kappa(s) | ) ds,
\end{equation*}
and by the Gau{\ss}--Bonnet Theorem (see Remark~\ref{rem:Gauss_Bonnet}) we obtain \eqref{eq:Kbound}
and (\ref{eq:rel_W0_Wgamma}).
We further deduce from (\ref{eq:AH}) that
\begin{equation}\label{eq:GaussBonnet}
\int_{\Omega} | A |_g^2 Q \, dx  =  \int_{\Omega} H^2 Q \, dx - 2 \int_{\Omega} K Q \, dx 
 =  4 W_0(u) + 2 \int_{\partial \graph(u)} \kappa_g ds - 4 \pi \chi(\graph(u)),
\end{equation}
and as above we deduce (\ref{eq:Abound}) in the case that $u \in C^2(\overline{\Omega})$. Finally suppose that $u \in H^2(\Omega)$ such that $u-\varphi \in H^1_0(\Omega)$. 
Then there exists a sequence 
$(u_k)_{k \in \mathbb{N}}$ such that $u_k \in C^2(\overline{\Omega}), u_k =\varphi$ on $\partial \Omega$ and $u_k \rightarrow u$ in 
$H^2(\Omega), k \rightarrow \infty$. We deduce from the generalised Lebesgue convergence theorem that
\begin{align*}
	\int_\Omega K_kQ_k\,dx \,&=\, \int_\Omega \frac{\det D^2u_k}{Q_k^3}\,
	\to\, \int_\Omega \frac{\det D^2u}{Q^3} \,=\, \int_\Omega K\, Q \,dx,
\end{align*}
since $\frac{\det D^2u_k}{Q_k^3}$ converges pointwise almost everywhere and since by $Q_k\geq 1$ we obtain that $|D^2u_k|^2$ is a $L^1$-convergent 
sequence of dominating functions. This yields \eqref{eq:Kbound} in the general case.

Since  (\ref{eq:Abound}) holds for $C^2$-functions we infer that
\begin{align*}
\int_{\Omega} | A_k |_{g_k}^2 Q_k \, dx  \leq  4 W_0(u_k) + 2 \bigl( \| \varphi \|_{W^{2,1}(\partial\Omega)}+ \| \kappa \|_{L^1(\partial\Omega)} \bigr) 
- 4 \pi \chi(\Omega) 
\end{align*}
and, with similar arguments as above, 
passing to the limit yields the result.
\end{proof}

%
\section{Sequences of graphs with bounded Willmore energy}
\label{sec:bounds}

\subsection{Area and diameter bounds}

The following celebrated diameter estimate of Leon Simon~\cite{Simon} is the starting point
of our reasoning.
\begin{theorem}[Lemma 1.2 in \cite{Simon}] \label{thm:diam}
 Let $\Gamma\subset \mathbb{R}^n$ be a smooth connected and compact surface 
with boundary.
Then there exists a constant $C$ which only depends on $n$ such that
\begin{displaymath}
\operatorname{diam}(\Gamma) \leq C \Bigl( \int_{\Gamma} | A|_g \, dS+ \sum_{j} \operatorname{diam}(\Gamma_j) \Bigr),
\end{displaymath}
where $\Gamma_j$ are the connected components of $\partial \Gamma$.
\end{theorem}

From now on we always work in $\mathbb{R}^2$.
The following result follows from the preceding estimate and is the key for establishing a-priori bounds on  sequences
which are bounded with respect to the  $W_0$-- or  $W_\gamma$--functional.

\begin{theorem}  \label{thm:aprioribounds}
Suppose that $u \in H^2(\Omega)$ satisfies $u-\varphi \in H^1_0(\Omega)$ Then there exists a constant $C$ that only depends
on $\Omega$ and $\Vert \varphi \Vert_{W^{2,1}(\partial \Omega)}$   such that
\begin{equation}  \label{eq:ubounds}
\displaystyle 
\sup_{x \in \Omega} | u(x)| + \int_{\Omega} Q \, dx \leq C \bigl( W_0(u)^2 +1 \bigr).
\end{equation}
\end{theorem}

\noindent
Examples~\ref{ex:singular_graph_1} and  \ref{ex:singular_graph_2} below show that it is not
possible to obtain uniform bounds in $W^{1,p}(\Omega)$ for any $1<p\le \infty$.
\begin{proof} 
Let us first assume that $u \in C^2(\overline{\Omega})$ and $u=\varphi$ on $\partial \Omega$.
A careful inspection of
the proof of  \cite[Lemma 1.2]{Simon} shows that the bound in 
Theorem~\ref{thm:diam} holds for $\graph(u)$ (see\cite{Gulyak}) so that
\begin{equation} \label{diamu}
\displaystyle
\operatorname{diam}(\graph(u))  \leq C \Bigl( \int_{\graph(u)} | A|_g \, dS+  \operatorname{diam}(\partial \graph(u)) \Bigr).
\end{equation}
We note that
\begin{displaymath}
\operatorname{diam}(\graph(u)) \geq \sup_{x,y \in \overline{\Omega}, x \neq y} | u(x) - u(y)| \geq 
\sup_{x \in \Omega, y \in \partial \Omega} | u(x) - u(y)| \geq
\sup_{x \in \Omega} | u(x)| -
\sup_{x \in \partial \Omega} | \varphi(x)|
\end{displaymath}
while  $\operatorname{diam}(\partial \graph(u)) \leq C\left( 1+  \Vert \varphi \Vert_{C^0(\partial \Omega)}\right)$ with a constant that depends on 
$\operatorname{diam}(\Omega)$. Hence we deduce  from (\ref{eq:Abound}) and (\ref{diamu}) that
\begin{eqnarray}  
\sup_{x \in \Omega} | u(x)|  & \leq &   C \Bigl( \int_{\Omega} | A|_g Q \, dx +  \Vert \varphi \Vert_{C^0(\partial \Omega)}  +1 \Bigr)  +
\Vert \varphi \Vert_{C^0(\partial \Omega)} \nonumber  \\
& \leq &  C \Bigl( \bigl( \int_{\Omega} | A |_g^2 Q \, dx \bigr)^{{1}/{2}} \bigl( \int_{\Omega} Q \, dx \bigr)^{{1}/{2}}
+ \Vert \varphi \Vert_{C^0(\partial \Omega)}  +1 \Bigr)  \label{eq:maxu}  \\
& \leq & C \bigl( W_0(u) +1) \bigr)^{ {1}/{2}} \bigl( \int_{\Omega} Q \, dx \bigr)^{{1}/{2}}  + C, \nonumber
\end{eqnarray}
where $C$ depends on the diameter and the topology of $\Omega$, 
$\Vert \varphi \Vert_{W^{2,1}(\partial \Omega)}$  and $\Vert \kappa \Vert_{L^1(\partial \Omega)}$.  

Our next aim is to bound $\int_{\Omega} Q \, dx$. We have
\begin{eqnarray*}
\int_{\Omega} uH \, dx & = & \int_{\Omega} u  \nabla \cdot \Bigl( \frac{ \nabla u}{Q} \Bigr) \, dx =
- \int_{\Omega} \frac{| \nabla u |^2}{Q} \, dx + \int_{\partial \Omega} \frac{u \frac{\partial u}{\partial \nu}}{Q} ds \\
& = & - \int_{\Omega} Q \, dx + \int_{\Omega} \frac{1}{Q} \, dx +  \int_{\partial \Omega} 
\frac{\varphi \frac{\partial u}{\partial \nu}}{Q} ds.
\end{eqnarray*}
This integration by parts is the place where we essentially exploit that the surface $\graph(u)$ is a graph.

Combining this relation with (\ref{eq:maxu}) we deduce
\begin{eqnarray*}
\int_{\Omega} Q \, dx & \leq &  | \Omega | + \Vert \varphi \Vert_{L^1(\partial \Omega)}  
       +  | \Omega |^{1/2}   \sup_{x \in \Omega} | u(x) | \bigl( \int_{\Omega} H^2 Q \, dx \bigr)^{{1}/{2}}  \\
& \leq & C + C \Bigl( \bigl( W_0(u)+1 \bigr)^{{1}/{2}}  \bigl( \int_{\Omega} Q \, dx \bigr)^{{1}/{2}} +1 \Bigr) W_0(u)^{{1}/{2}} \\
& \leq & \frac{1}{2} \int_{\Omega} Q \, dx + C \bigl( W_0(u)^2+1 \bigr).
\end{eqnarray*}
Inserting this estimate into (\ref{eq:maxu}) yields (\ref{eq:ubounds}) for
$u \in C^2(\overline{\Omega}), u=\varphi$ on $\partial \Omega$. The general case is  obtained with the help of an approximation argument as in Lemma \ref{lemma:Abounds}.
\end{proof}

\subsection{Examples: No higher integrability of gradients and singular graphs with finite Willmore energy}\label{sec:no_lower_bound}
In this section we present some illustrative examples. We demonstrate that the Willmore energy of a function $u$ does not 
control any $L^p$-norm, $p>1$, of $\nabla u$. Furthermore, we give examples of functions $u$ that are only in $BV(\Omega)\setminus W^{1,1}(\Omega)$ 
but for which $\graph(u)$ describes a smooth surface. These functions can be approximated in $L^1(\Omega)$ by smooth functions with uniformly bounded Willmore energy. 
In particular, the estimates on diameter and area obtained in Theorem \ref{thm:aprioribounds} are in this sense optimal, and sequences with uniformly bounded Willmore 
energy may $L^1$-converge to limit functions that are not even in $W^{1,1}(\Omega)$.

In order to construct appropriate examples it is well known that $\log\circ \log$
is a good ingredient, see for example \cite{Freh73,HuMe86,Toro94}, and that in particular $H^2(\Omega)\not\hookrightarrow W^{1,\infty} (\Omega)$. 
But we even can show a bit more: In spite of the 
non-homogeneous form of the Willmore energy, we may have unbounded gradients and arbitrarily
small  Willmore energy at the same time.

\begin{example}[\cite{Toro94}]\label{ex:singular_graph_1} 
 Let $\Omega=B:=B_1(0)$ be the unit disk. We consider $u:\overline{B}\to \mathbb{R}$, 
which is smooth in $\overline{B}\setminus \{ 0\}$, satisfies homogeneous Dirichlet boundary conditions 
$u=\partial_\nu u=0$ on $\partial B$,
and $u(x)=x^1 \log(|\log(r)|)$ for $r=|x|$ close to $0$. Then, close to $0$ we have 
\begin{eqnarray*}
 |\nabla u|(x) &=&\left|  \log(|\log(r) |) \right| +O(1),\qquad 
 |D^2u|(x) \,=\, O(\frac{1}{r|\log r|})
\end{eqnarray*}
and therefore $u\in H^2_0 (B) \setminus W^{1,\infty}(B) $.

For $\varepsilon\downarrow 0$  we consider $\varepsilon u$:
$$
H[\varepsilon u] =\varepsilon\frac{\Delta u}{(1+\varepsilon^2 |\nabla u|^2 )^{1/2}}
  -\varepsilon^3\frac{\nabla u \cdot D^2u \cdot (\nabla u)^T}{(1+\varepsilon^2 |\nabla u|^2 )^{3/2}}.
$$
Up to a factor $2$, a majorising function for $H[\varepsilon u] ^2 \sqrt{1+\varepsilon^2 |\nabla u|^2 }$ is
given by
\begin{align*}
&\varepsilon^2\frac{(\Delta u)^2}{(1+\varepsilon^2 |\nabla u|^2 )^{1/2}}
  +\varepsilon^6\frac{|\nabla u|^4 \cdot |D^2 u |^2 }{(1+\varepsilon^2 |\nabla u|^2 )^{5/2}}
\le \varepsilon^2 (\Delta u)^2 + \varepsilon^2\frac{ |D^2 u|^2 }{(1+\varepsilon^2 |\nabla u|^2 )^{1/2}}
\le C  \varepsilon^2 |D^2 u|^2.
\end{align*}
Hence
$$
\lim_{\varepsilon\downarrow 0} W_0(\varepsilon u) =0,
$$
while at the same time 
$$
\forall \varepsilon >0: \sup_{x\in B} \varepsilon |\nabla u (x) | = +\infty.
$$
This means also that even for the trivial Willmore surface $(x,y)\mapsto 0$ we find a minimising 
sequence with unbounded gradients. 
\end{example}


Next we give an example of a function $u\in W^{1,1}(\Omega)\setminus H^2 (\Omega)$ such that $\graph(u)$ is smooth as a surface. In this example the singularity is 
purely analytical, introduced by the specific choice of parametrisation as graph. This example shows further that for $p>1$, no $W^{1,p}$--norm may be
estimated in terms of the Willmore energy.

\begin{example}\label{ex:singular_graph_2}
We choose an odd integer $k\in 2\mathbb{N}+1$ larger than or equal to $3$.
We   consider a nonincreasing function
$h\in C^0([0,2])\cap C^\infty([0,1)\cup(1,2],[-1,1])$ with
$$
h(r)=\left\{\begin{array}{ll}
               1\qquad & \mbox{\ for\ } r\in[0,1/2],\\
              \operatorname{sgn}(1-r) |1-r|^{1/k} & \mbox{\ for\ } r\in[3/4,5/4],\\
             -1\qquad & \mbox{\ for\ } r\in[3/2,2],\\
            \end{array}
\right.
$$
As a curve $r\mapsto (r,h(r))$  in $\mathbb{R}^2$, it is $C^\infty$--smooth,
because close to $1$, $h$ is the inverse of the analytic function $h\mapsto 1-h^k$.
On the other hand,
as a graph, close to $1$ the singularity of $h'$ is of order $ |1-r|^{-1+1/k}$
and the singularity of $h''$ of order $ |1-r|^{-2+1/k}$. This means that $h$ has a weak first 
derivative but not a weak second derivative.

The same applies to the graph of the radially symmetric function 
$u:\overline{B_2(0)}\to [-1,1]$, $u(x^1,x^2)=h(| (x^1,x^2) |)$ and
yields that $u\in W^{1,1}(B_2(0))\setminus H^2 (B_2(0))$. 
We even have that $u\not\in W^{1,k/(k-1)}(B_2(0))$. Observe that $\frac{k}{k-1}$ may become arbitrarily close to $1$.
However, since $\graph(u)$ is 
compact and smooth as a surface, its Willmore energy is well defined and finite.
\end{example}

\begin{example}\label{ex:singular_graph_2a}
It is also possible to introduce in the previous example a vertical piece and to obtain a surface that is not a graph, but that can be approximated by smooth graphs 
with uniformly bounded Willmore energy: Cut the surface in Example~\ref{ex:singular_graph_2}
along the circle where $u$ has  infinite slope and insert there a cylindrical part.
This can certainly be $L^1$-approximated by
a sequence $(u_k)_{k\in\mathbb{N} }\subset \MDir$ with uniformly bounded Willmore energy, but $u\in BV (\Omega)
\setminus W^{1,1} (\Omega)$. Note that the limit of the graph functions is a $BV$-function with a non-vanishing jump part and that at its jump points 
also the absolutely continuous part $\nabla^a u$ of the gradient blows up.
\end{example}

In the next example a function $u\in W^{2,1}(\Omega)\setminus H^2(\Omega)$ is constructed such that its graph has bounded Willmore energy. Here the singularity 
is independent of the parametrisation of $\graph(u)$ and therefore a ``real'' geometric singularity.
\begin{example}\label{ex:singular_graph_3}
 We consider $\Omega=B:=B_1(0)$ and a function $u\in C^0_0 ( {\Omega})\cap C^\infty (\overline{\Omega}\setminus \{0\})$ 
such that for $r=|x|$ close to $0$:
\begin{eqnarray*}
u(x)&=& x^1 \, (-\log r)^{1/2},\\ 
u_{x^1}(x)&=& (-\log r)^{1/2}- \frac{(x^1)^2}{2r^2} (-\log r)^{-1/2}=(-\log r)^{1/2}+O(1),\\ 
u_{x^2}(x)&=& - \frac{x^1\, x^2}{2r^2} (-\log r)^{-1/2}= O(1),\\ 
u_{x^1x^1}(x)&=&-\frac{3x^1}{2r^2}(-\log r)^{-1/2}+\frac{(x^1)^3}{r^4}(-\log r)^{-1/2}
                 - \frac{(x^1)^3}{4 r^4}(-\log r)^{-3/2},\\
u_{x^1x^2}(x)&=&-\frac{x^2}{2r^2}(-\log r)^{-1/2}+\frac{(x^1)^2 x^2}{r^4}(-\log r)^{-1/2}
                 - \frac{(x^1)^2 x^2}{4 r^4}(-\log r)^{-3/2},\\
u_{x^2x^2}(x)&=&-\frac{x^1}{2r^2}(-\log r)^{-1/2}+\frac{x^1(x^2)^2 }{r^4}(-\log r)^{-1/2}
                 - \frac{x^1 (x^2)^2 }{4 r^4}(-\log r)^{-3/2}.
\end{eqnarray*}
Concerning the asymptotic behaviour for $r\downarrow 0$ we have
in view of $H[ u] =\frac{\Delta u}{(1+ |\nabla u|^2 )^{1/2}}
  -\frac{\nabla u \cdot D^2u \cdot (\nabla u)^T}{(1+ |\nabla u|^2 )^{3/2}}$
and  $|H|\le C \frac{|D^2 u|}{Q}$:
\begin{eqnarray}
 |\nabla u | &=& (-\log r)^{1/2}+O(1),\qquad   Q=  (-\log r)^{1/2}+O(1),\nonumber\\ 
|D^2 u|&=& O\left(\frac{1}{r\sqrt{ (-\log r)}}\right),\qquad  |D^2 u|^2= O\left(\frac{1}{r^2 |\log r| }\right),\label{eq:F.2}\\ 
H^2 Q &\le & C \frac{|D^2 u|^2}{Q}=O\left(\frac{1}{r^2 |\log r|^{3/2} }\right)\in L^1 (B_{1/2}(0)).\label{eq:F.3}
\end{eqnarray}
The second derivatives of $u$ are locally integrable around $0$ and so, exist as weak derivatives in the
whole domain $B$, such that we even have $u\in W^{2,1}(B)\subset W^{1,1}(B) $.
Thanks to (\ref{eq:F.3}) we see further  that $W_0(u)<\infty$. However, $u\not\in H^2 (B)$. To this end we show that 
$u_{x^1x^1}\not\in L^2(B)$. We observe that
$$
x\mapsto - \frac{(x^1)^3}{4 r^4}(-\log r)^{-3/2}\in L^2 (B_{1/2}(0)),
$$
while 
$$
x\mapsto\left| - \frac{3x^1}{2r^2}(-\log r)^{-1/2}+\frac{(x^1)^3}{r^4}(-\log r)^{-1/2}\right|^2
=\frac{(x^1)^2}{r^4(-\log r)} \left(\frac{3}{2}-\frac{(x^1)^2}{r^2} \right)^2
\ge \frac{(x^1)^2}{4r^4(-\log r)}.
$$
The latter function is not in $ L^1 (B_{1/2}(0))$. Otherwise, one would have also that
$x\mapsto \frac{(x^2)^2}{4r^4(-\log r)}\in L^1 (B_{1/2}(0))$ and so that
$x\mapsto \frac{(x^1)^2 + (x^2)^2}{4r^4(-\log r)}=\frac{1}{4r^2(-\log r)}\in L^1 (B_{1/2}(0))$,
a contradiction. This shows that (\ref{eq:F.2}) displays the precise asymptotic behaviour of $D^2u$
close to $0$.

Finally, $\graph(u)$ is not a $C^2$--smooth surface because the curvature of the curve $t \mapsto (t ,0,u(t,0))$
is given by
$$
\frac{u_{x^1x^1}(t,0) }{(1+u_{x^1}(t,0)^2)^{3/2}}
=\frac{2-\frac{1}{\log|t|}}{-4t \sqrt{-\log|t|} (-\log|t| -\frac{1}{4\log|t| })^{3/2} }
=\frac{1}{-4t  (\log|t|)^2} \cdot \frac{2-\frac{1}{\log|t|} }{(1+\frac{1}{4(\log|t|)^2})^{3/2} }
$$
and becomes unbounded and so undefined  for $t\downarrow 0$.
\end{example}

\section{Compactness and lower bounds for energy-bounded sequences}
\label{sec:lowersemicontinuous}

When we consider minimising sequences for the Willmore functional of graphs (subject to appropriate boundary conditions), or more generally sequences 
with uniformly bounded Willmore energy Theorem~\ref{thm:aprioribounds} shows that $BV\cap L^\infty$ is a natural space, where uniform bounds hold. 
In particular, such sequences are precompact in $L^1$. For this reason it is useful to study the behaviour of the Willmore functional with respect to $L^1$-convergence. 
However, as the examples of the previous section indicate, limit points need not  remain in $H^2(\Omega)$ and even can have an $L^1$-limit with jump discontinuities, 
which results in vertical parts in the boundary of the corresponding sublevel-sets. This leads to substantial difficulties in the analysis.
Nevertheless,
we derive below some additional (mild) regularity properties and control the Willmore energy of the absolutely continuous part of limit configurations. 

To introduce an appropriate generalised formulation consider for $u\in BV(\Omega)$ the absolutely continuous part $\nabla^a u\in L^1(\Omega)$ of the $\R^2$-valued  
measure $\nabla u$, set $Q^a := \sqrt{1+|\nabla^a u|^2}$ and define the absolutely continuous contribution to the Willmore energy as
\begin{align}
	\Wnv_0(u) \,:=\, 
	\begin{cases}
		\frac{1}{4}\int_\Omega \big(\nabla\cdot \frac{\nabla^a u}{Q^a}\big)^2 Q^a \,dx \quad &\text{ if }\frac{\nabla^a u}{Q^a}\in H(\dive,\Omega) \text { and the integral is finite},\\
		\infty &\text{ else,}
	\end{cases} \label{eq:def-Wnv}
\end{align}
where the space $H(\operatorname{div},\Omega)$ of $L^2(\Omega,\mathbb{R}^2)$--vector fields with weak 
divergence in $L^2(\Omega)$ was introduced in Section \ref{sec:prelim}.

In the next theorem we prove our main lower bound and compactness results. For energy-bounded 
sequences in $H^2(\Omega)$ that satisfy a suitable  boundary condition 
we show that there exists a $L^1$-convergent subsequence. The limit belongs to $BV(\Omega)\cap L^\infty(\Omega)$, 
the absolutely continuous contribution to the Willmore energy $\Wnv_0$ is finite 
and obeys an  estimate from above. In particular, if the limit is already a $W^{1,1}(\Omega)$-function the full Willmore functional is controlled. 
This also shows the $L^1$-lower semicontinuity of $W_0$ in $H^2(\Omega)$ (subject to prescribed   boundary conditions).

\begin{theorem}\label{thm:L1lowersemicontinuity}
Let $(u_k)_{k\in\N}$ be a given sequence in $H^2(\Omega)$ that satisfies $u_k-\varphi \in H^1_0(\Omega)$ for all $k\in\N$ and
\begin{align}
	\liminf_{k\to\infty} W_0(u_k)\,<\,\infty. \label{eq:bound-W0}
\end{align}
Then there exists a function $u\in BV(\Omega)\cap L^\infty(\Omega)$ with $\frac{\nabla^a u}{Q^a}\in H(\dive,\Omega)$ such that after passing to a subsequence
\begin{equation}\label{eq:3.2}
 u_k\to u \text{ in } L^1(\Omega)\quad (k\to\infty)
\end{equation}
and  
\begin{equation}\label{eq:3.3-a}
 \Wnv_0(u) \le \liminf_{k\to\infty}W_0(u_k).
\end{equation}
In particular, if $u\in H^2(\Omega)$ then
\begin{equation}\label{eq:3.3}
 W_0(u) \le \liminf_{k\to\infty}W_0(u_k).
\end{equation}
\end{theorem}
Moreover, in a sense made precise in Proposition~\ref{prop:Wgamma}, it is proved there that the trace of $u$ on $\partial\Omega$ satisfies 
$\Ha^{1}$-almost everywhere on $\{(Q^a)^{-1}>0\}\cap\partial\Omega$ the  boundary condition $u=\varphi$.

\medskip
A related lower semicontinuity result in the context of integral currents
was proved by Sch\"atzle \cite{Schaetzle2}. The area bound 
required there is in our case satisfied thanks to Theorem~\ref{thm:aprioribounds}, and therefore \eqref{eq:3.3} could also be deduced (with some additional work) 
from \cite[Theorem 5.1]{Schaetzle2}.
However, we prefer to give a self-consistent proof within the context of graphs, with the advantage that more elementary arguments apply, compared to Sch\"atzle's approach.

\begin{proof} 
There exists a subsequence, again denoted by
$(u_k)_{k \in \mathbb{N}}$,  and a constant $M \geq 0$ such that 
\begin{equation} \label{eq:3.1}
\displaystyle
W_0(u_k) \rightarrow \liminf_{k\to\infty}W_0(u_k) \quad \mbox{ and } \quad W_0(u_k) \leq M\,\text{ for all } k \in \mathbb{N}.
\end{equation}
Theorem \ref{thm:aprioribounds} implies that
\begin{equation}\label{eq:3.4}
 \| u_k \|_{C^0 (\overline\Omega)}\le C, \qquad \int_{\Omega} Q_k\, \, dx \le C \qquad \mbox{uniformly in } k \in \mathbb{N}.
\end{equation}
By the compactness theorem in BV \cite[Theorem 3.23]{AmFP00} we deduce that there exists a function $u\in BV(\Omega)$ and a subsequence $k\to\infty$ 
such that (\ref{eq:3.2}) holds. By \eqref{eq:3.4} we also have $u\in L^\infty(\Omega)$ and, possibly after passing to another subsequence, we obtain
\begin{equation}\label{eq:3.4.a}
 u_k \to u \mbox{\ strongly in } L^p(\Omega)\text{ for any }1\le p<\infty,\quad\text{ and a.e.~in }\Omega.
\end{equation}
Furthermore,  we deduce from (\ref{eq:Hessian_vs_secondff}), (\ref{eq:Abound}), and (\ref{eq:3.1}) that
\begin{equation}\label{eq:3.5}
 \int_{\Omega} \frac{| D^2 u_k |^2}{Q_k^5}\,dx\le  \int_{\Omega} | A_k |_g^2 Q_k\, \, dx \le C.
\end{equation}
Here we made use of the boundary condition $u_k-\varphi  \in H^1_0(\Omega)$.
We consider the bounded mappings
$$
q_k:=\frac{1}{Q_k^{5/2} }=\frac{1}{(1+| \nabla u_k |^2 )^{5/4} },\qquad   
v_k :=q_k \nabla u_k =\frac{1}{Q_k^{5/2} } \nabla u_k =\frac{1}{(1+| \nabla u_k |^2 )^{5/4} }\nabla u_k.
$$
We have
\begin{eqnarray*}
 \partial_i q_k &=& -\frac{5}{2} \sum_{\ell=1}^2\frac{\partial_\ell u_k\, \partial_i\partial_\ell u_k}{(1+| \nabla u_k |^2 )^{9/4} }
=O(\frac{| D^2 u_k |}{(1+| \nabla u_k |^2 )^{7/4} })= O(\frac{| D^2 u_k |}{Q_k^{7/2} }),\\
 \partial_i v_k &=&( \partial_i q_k)  \nabla u_k+ q_k\,  \partial_i  \nabla u_k = O(\frac{| D^2 u_k |}{Q_k^{5/2} }).
\end{eqnarray*}
By  (\ref{eq:3.5}) one has uniform boundedness of $(q_k)_{k\in\mathbb{N}}$ and  $(v_k)_{k\in\mathbb{N}}$ in $H^1(\Omega)$. Hence one
finds $q,v\in H^1(\Omega)$ such that, after passing to a subsequence,
\begin{align}
	q_k &\rightharpoonup q ,\quad v_k \rightharpoonup v \quad\mbox{\ in\ } H^1(\Omega),\label{eq:w-conv-qk}\\
	q_k &\to q, \quad v_k \to v \quad\mbox{\ in any\ } L^p(\Omega), 1\leq p<\infty, \mbox{\ and almost everywhere in }\Omega.	\label{eq:ptw-conv-qk}
\end{align}
From now on we fix precise representatives for $q_k,v_k$, $k\in\N$ and $q,v$. By \cite[Theorem 7, Section 1.C]{Evan90} 
there exists a  subsequence $k\to\infty$ and
for every $m\in\N$  an open set $E_m\subset\Omega$ with $\capa_{3/2}(E_m)\leq \frac{1}{m}$ (the choice of $\frac{3}{2}$ here and 
in the following is for convenience, any exponent in $(1,2)$ instead of $\frac{3}{2}$ works as well) such that
\begin{align}
	q_k \to q ,\quad v_k \to v \quad\mbox{\ uniformly in\ } \Omega\setminus E_m. \label{eq:qkvk-Cap}
\end{align}
This yields for $E=\cap_{m\geq 1} E_m$ that 
\begin{align}
	q_k &\to q, \quad v_k \to v \quad\text{ pointwise in } \Omega\setminus E.	\label{eq:qe-conv-qk}
\end{align}
Since $\capa_{3/2}(E)\le \capa_{3/2}( E_m)$ for all $m\in\N$ by \cite[Remark in Section 4.7.1]{EvGa92} we conclude that $E$ has $\frac{3}{2}$-capacity zero and thus satisfies 
$\Ha^1(E)=0$, see \cite[Theorem 4.7.4]{EvGa92}.

Due to the uniform area bound \eqref{eq:3.4}, \eqref{eq:qe-conv-qk} and Fatou's Lemma we deduce that
\begin{align}
	C\,\geq\, \liminf_{k\to\infty} \int_\Omega Q_k \,dx\,\geq\,  \int_\Omega \liminf_{k\to\infty}(q_k)^{-2/5}\,dx \,=\, \int_\Omega q^{-2/5}\,dx. \label{eq:Qa-bound-pre}
\end{align}
This shows that $q^{-2/5}\in L^1(\Omega)$, in particular
\begin{equation}\label{eq:3.6}
 q>0 \mbox{\ almost everywhere in\ }\Omega.
\end{equation}

We next claim that
\begin{equation}\label{eq:3.7-new}
	v\,\LL^2 \,=\, q\nabla u\quad\text{ as Radon measures on } \Omega.
\end{equation}
To prove this, consider any $\eta \in C^\infty_0 (\Omega, \mathbb{R}^2)$. Making use of (\ref{eq:3.4.a}) and so in particular
of the $C^0$-bounds for $u_k$ we obtain:
\begin{eqnarray}
 \int_\Omega \eta \cdot v\, dx &=&  \int_\Omega \eta \cdot v_k\, dx +o(1)
= \int_\Omega q_k \eta \cdot\nabla u_k \, dx +o(1) \notag\\
&=& - \int_\Omega (\operatorname{div}\eta ) \underbrace{q_k }_{\to q \mbox{\scriptsize\ in\ }L^2}\underbrace{u_k }_{\to u \mbox{\scriptsize\ in\ }L^2}\, dx 
- \int_\Omega  \underbrace{(\eta \cdot \nabla q_k) }_{O(1) \mbox{\scriptsize\ in\ }L^2} \cdot \underbrace{(u_k-u)}_{\to 0 \mbox{\scriptsize\ in\ }L^2}\, dx  \notag\\
&&- \int_\Omega (\eta u)\cdot \underbrace{\nabla q_k }_{\rightharpoonup \nabla q \mbox{\scriptsize\ in\ }L^2}\, dx  +o(1) \notag\\
&=&-\int_\Omega( \operatorname{div}\eta ) qu\, dx -\int_\Omega(\eta \cdot \nabla q) \, u\, dx +o(1)  \notag\\
 &=& -\int_\Omega u\nabla\cdot (q\eta)  \, dx  +o(1). \label{eq:3.7-1}
\end{eqnarray}
For the right-hand side we claim that
\begin{align}
	-\int_\Omega u \nabla\cdot \big(\eta q\big) \,dx \,=\, \int_\Omega \eta q\cdot d(\nabla u).  \label{eq:3.7-2}
\end{align}
In fact we can approximate $\eta q$ strongly in $H^1_0(\Omega)$ by a smooth sequence $(w_\ell)_{\ell\in\N}$ that is uniformly bounded in $C^0(\overline{\Omega})$.
As above we deduce that there exists a set $\tilde{E}\subset\Omega$ with $\frac{3}{2}$-capacity zero and thus (see above) with $\Ha^1(\tilde{E})=0$, such that 
$w_\ell\to \eta q$ everywhere in $\Omega\setminus \tilde{E}$. Since $|\nabla u|(\tilde{E})=0$ by \cite[Lemma 3.76]{AmFP00} we have $w_\ell\to \eta q$ in 
$|\nabla u|$-almost every point and deduce from Lebesgue's dominated convergence theorem that
\begin{align*}
	-\int_\Omega  u \nabla\cdot \big(\eta q\big)\,dx \,=\, -\lim_{\ell\to\infty} \int_\Omega  u\nabla\cdot w_\ell \,dx	\,=\, \lim_{\ell\to\infty}  \int_\Omega w_\ell\cdot d(\nabla u) 
	\,=\, \int_\Omega \eta q\cdot d(\nabla u),
\end{align*}
which proves \eqref{eq:3.7-2}. From \eqref{eq:3.7-1}, \eqref{eq:3.7-2} we deduce \eqref{eq:3.7-new}. We next show 
\begin{align}
	 v \,&=\, q\nabla^a u\quad\text{ almost everywhere in } \Omega, \label{eq:3.7}\\
	 q\nabla^s u  \,&=\, 0\qquad\text{ as Radon measures on }\Omega. \label{eq:nablau-sing}
\end{align}
To prove these properties we first obtain from \cite[Theorem 1.3.5]{Simo83} that for every $k\in\N$ there exists a set $B_k$ with $|B_k|=0$ such that for any 
$x_0\in \{|\nabla^a u|\leq k\}\setminus B_k$
\begin{align*}
	\lim_{r\downarrow 0} \frac{|B_r(x_0)\cap\{|\nabla^a u|>k\}|}{|B_r(x_0)|} \,=\, 0.
\end{align*}
This implies that for almost every $x_0\in\Omega$
\begin{align}
	\lim_{r\downarrow 0} \frac{|B_r(x_0)\cap\{|\nabla^a u|>|\nabla^a u|(x_0)+1\}|}{|B_r(x_0)|} \,=\, 0. \label{eq:40-1}
\end{align}
Next we deduce from \cite[Corollary 5.11]{Magg12} and \cite[Theorem 1.6.3, Corollary 1.7.1]{EvGa92} that almost all $x_0\in\Omega$ are Lebesgue points of both 
$\nabla^a u$ and $q$, and that $\nabla u$ is absolutely continuous in almost every $x_0$, i.e.
\begin{align}
	\lim_{r\downarrow 0} \fint_{B_r(x_0)} d(\nabla u) \,&=\, \nabla^a u(x_0)\in\R^2, \label{eq:40-21}\\
	\lim_{r\downarrow 0} \frac{|\nabla^s u|(B_r(x_0))}{|B_r(x_0)|} \,&=\, 0, \label{eq:40-2}\\
	\lim_{r\downarrow 0} \fint_{B_r(x_0)} |\nabla^a u(x) -\nabla^a u(x_0)|\,dx \,&=\, 0, \label{eq:40-3}\\
	\lim_{r\downarrow 0} \fint_{B_r(x_0)} |q(x) -q(x_0)|\,dx \,&=\, 0. \label{eq:40-4}
\end{align}
To prove \eqref{eq:3.7} we therefore can restrict ourselves to $x_0\in\Omega$ such that \eqref{eq:40-1}-\eqref{eq:40-4} are satisfied. We then compute
\begin{align*}
	&\Big| \fint_{B_r(x_0)} q\,d(\nabla u) - q(x_0)\fint_{B_r(x_0)}d(\nabla u) \Big|\\
	\leq\,
	&\fint_{B_r(x_0)} |q(x)-q(x_0)|\,d|\nabla u|(x)\\
	\leq\, &\fint_{B_r(x_0)} |q(x)-q(x_0)||\nabla^a u|(x)\,dx + \fint_{B_r(x_0)}d|\nabla^s u| \\
	\leq\, &(|\nabla^a u|(x_0)+1)\fint_{B_r(x_0)} |q(x)-q(x_0)|\,dx \\
        &\quad + \frac{1}{ |B_r(x_0)| }\int_{B_r(x_0)\cap\{|\nabla^a u|>|\nabla^a u|(x_0)+1\}}|\nabla^a u(x)-\nabla^a u(x_0)|\,dx\\
	&\quad + \frac{1}{ |B_r(x_0)| }\int_{B_r(x_0)\cap\{|\nabla^a u|>|\nabla^a u|(x_0)+1\}}|\nabla^a u(x_0)|\,dx+ \fint_{B_r(x_0)}d|\nabla^s u|\\
	\to\,& 0\quad\text{ for }r\downarrow 0
\end{align*}
by \eqref{eq:40-4},\eqref{eq:40-3},\eqref{eq:40-1} and \eqref{eq:40-2}. This shows $(q\nabla u)^a(x_0)=q(x_0)\nabla^a u(x_0)$ and implies by \eqref{eq:3.7-new} 
that almost everywhere in $\Omega$
\begin{align*}
	v \,=\, (q\nabla u)^a \,=\, q \nabla^a u
\end{align*}
holds, which gives \eqref{eq:3.7}. We further deduce from \eqref{eq:3.7-new} that
\begin{align*}
	0\,=\, (q\nabla u)^s \,=\, q\nabla u - (q\nabla u)^a \,=\, q\nabla u - q\nabla^a u \,=\, q\nabla^s u,
\end{align*}
and therefore \eqref{eq:nablau-sing} holds.

Finally, we claim that there exists a set $E_1$ of $\frac{3}{2}$-capacity zero such that $\nabla^a u$ has an approximately continuous representative on $\{q>0\}\setminus E_1$ 
that satisfies
\begin{equation}\label{eq:3.7-new2}
	\nabla^a u \,=\, q^{-1}v \quad\text{ in } \{q>0\}\setminus E_1.
\end{equation}
In fact, by \cite[Theorem 4.8.1]{EvGa92} first there exists a set $E_1$ of $\frac{3}{2}$-capacity zero such that $\Omega\setminus E_1$ only consists of Lebesgue points of both 
$q$ and $v$. Therefore $q,v$ are approximately continuous in $\Omega\setminus E_1$ and by \eqref{eq:3.7} and the properties of approximate continuity stated in 
Section \ref{sec:prelim} we see that the approximately continuous representative of $\nabla^a u$ is well-defined in $\{q>0\}\setminus E_1$ and that \eqref{eq:3.7-new2} holds.

Enlarging the set $E$ from \eqref{eq:qe-conv-qk} by $E_1$ we conclude that
\begin{align}
 \nabla u_k =\frac{1}{q_k} v_k \,&\to\,  \frac{1}{q} v =\nabla^a u\qquad \quad  &&\text{ in }\{q>0\}\setminus E,\label{eq:3.8}\\
Q_k =\sqrt{1+|\nabla u_k |^2 }\,&\to\,  \sqrt{1+|\nabla^a u |^2 }  =Q^a \quad  &&\text{ in }\{q>0\}\setminus E, \label{eq:3.9}
\end{align}
where $\capa_{3/2}(E)=0$. Using \eqref{eq:3.6} we deduce that
\begin{align*}
	\int_\Omega Q^a\,dx \,\leq\, \liminf_{k\to\infty} \int_\Omega Q_k\,dx\,\leq\, C.
\end{align*}
We next discuss convergence properties of the mean curvatures $H_k =\operatorname{div} \left( \frac{\nabla u_k}{Q_k} \right)$, $k\in\N$.
In view of  (\ref{eq:3.1}) we have that
$$
\int_\Omega H_k^2\, dx \le \int_\Omega H_k^2Q_k \, dx\le 4M.
$$
Hence there exists  ${H}^a \in L^2 (\Omega )$ such that after passing to a subsequence 
$$
	H_k \rightharpoonup {H}^a \mbox{\ in\ }L^2 (\Omega ).
$$
By Lebesgue's theorem we further deduce for any $\zeta \in C^\infty_0 (\Omega)$
\begin{eqnarray*}
 \int_\Omega {H}^a \zeta \, dx &=& \lim_{k\to\infty} \int_\Omega H_k \zeta \, dx 
=-  \lim_{k\to\infty} \int_\Omega \frac{\nabla u_k}{Q_k }\cdot \nabla \zeta = -  \int_\Omega\frac{\nabla^a u}{Q^a}\cdot \nabla \zeta \, dx,
\end{eqnarray*}
where we have used that $\frac{\nabla u_k}{Q_k }$ is uniformly bounded and converges pointwise a.e.~to $\frac{\nabla^a u}{Q^a}$ 
by \eqref{eq:3.6}, (\ref{eq:3.8}) and (\ref{eq:3.9}). This shows that $\frac{\nabla^a u}{Q^a}\in H(\dive,\Omega)$ and that
\begin{align*}
	\dive \frac{\nabla^a u}{Q^a} \,=\, H^a\quad\text{ weakly.}
\end{align*}
We next claim that even
\begin{equation}\label{eq:3.10}
H_k\sqrt{Q_k}\rightharpoonup H^a\sqrt{Q^a} \quad\mbox{\ in\ } L^2(\Omega ).
\end{equation}
By \eqref{eq:3.1} there exist  $f\in L^2 (\Omega )$ such that after passing to a subsequence 
$$
H_k \sqrt{Q_k} \rightharpoonup f\quad\text{ in }L^2(\Omega).
$$
Moreover, we have for any $\zeta \in C^\infty_0 (\Omega)$ that $\left(1-\frac{\sqrt{Q^a}}{\sqrt{Q_k}} \right) \zeta \to 0$ almost
everywhere and so, by Lebesgue's theorem and $\frac{\sqrt{Q^a}}{\sqrt{Q_k}} \le \sqrt{Q^a}$, in $L^2 (\Omega )$. Hence
\begin{eqnarray*}
  \int_\Omega \zeta \left( H_k\sqrt{Q_k} - H^a\sqrt{Q^a}\right) \, dx &=&  \int_\Omega \underbrace{H_k \sqrt{Q_k}}_{\mbox{\scriptsize $O(1)$ in $L^2$}}
\underbrace{  \left(1-\frac{\sqrt{Q^a}}{\sqrt{Q_k}} \right) \zeta  }_{\mbox{\scriptsize $\to 0$ in $L^2$}}\, dx
+\int_\Omega\underbrace{\zeta \sqrt{Q^a}  }_{\in L^2} \underbrace{(H_k -H^a) }_{\mbox{\scriptsize $\rightharpoonup 0$ in $L^2$}}\, dx\\
&\to & 0 \mbox{ for } k\to \infty.
\end{eqnarray*}
We conclude that for all $\zeta \in C^\infty_0 (\Omega)$
$$
 \int_\Omega \zeta \left(f - H^a\sqrt{Q^a} \right)\, dx=\lim_{k\to\infty}\int_\Omega \zeta \left( H_k \sqrt{Q_k}- H^a\sqrt{Q^a} \right)\, dx=0.
$$
This proves  that $f=H^a\sqrt{Q^a}$ and so finally \eqref{eq:3.10}.

The weak lower semicontinuity of the $L^2$-norm eventually yields
$$
\Wnv_0(u)= \frac{1}{4} \int_\Omega \left( H^a\right)^2 Q^a\, dx \le \frac{1}{4} \liminf_{k\to\infty }\int_\Omega H_k^2 Q_k\, dx=\liminf_{k\to\infty }W_0(u_k)
$$
as claimed.
\end{proof}


We next show that in $H^2(\Omega)$ subject to a suitable  boundary condition we have continuity of the total Gau\ss-curvature with respect to $L^1$-convergence. 
Together with \eqref{eq:rel_W0_Wgamma} and Theorem~\ref{thm:L1lowersemicontinuity} this implies lower semicontinuity as in \eqref{eq:3.3} also for 
$W_\gamma$, $\gamma \in \mathbb{R}$ arbitrary.

\begin{proposition}\label{prop:L1lowersemicontinuity}
Suppose that $\Omega $ is $C^3$--smooth
and that $\varphi\in C^3(\overline{\Omega}) $.

Let $u_k, u \in H^2(\Omega)$ satisfy $u_k-\varphi \in H^1_0(\Omega)$. 
Let $K_k$ and $K$, resp., denote the Gau\ss\ curvatures  of their graphs.
Then
\begin{equation*}
 u_k\to u \quad\text{\ in\ } L^1(\Omega),\qquad \sup_{k\in\N} W_0(u_k)\,<\,\infty
\end{equation*}
implies that
 \begin{equation}\label{eq:3.11}
 \int_\Omega K Q\, dx = \lim_{k\to\infty} \int_\Omega K_k Q_k\, dx.
\end{equation}
\end{proposition}

\begin{proof}
 We shall first collect a number of equivalent representations for the total curvature which are
convenient in different situations we have to deal with. We use the notation from Remark~\ref{rem:geod_curv} and assume the boundary $\partial\Omega$
to be parametrised by arclength, $\varphi'(s)$ and $\varphi''(s)$ have to be understood correspondingly.
According to Remarks~\ref{rem:geod_curv} and  \ref{rem:Gauss_Bonnet}
and using an approximation argument  we have
\begin{eqnarray}
   \int_{\Omega} K Q \, dx &=&2 \pi \chi(\Omega)-\int_{\partial \graph(u)} \kappa_g ds
= 2 \pi \chi(\Omega)- \int_{\partial \Omega}\frac{-u_\nu(s) \varphi''(s)+\kappa (s) (1+\varphi'(s)^2)}{(1+\varphi'(s)^2+u_\nu (s)^2)^{1/2} (1+\varphi'(s)^2 )}\, ds \nonumber\\
&=& 2 \pi \chi(\Omega) + \int_{\partial \Omega}\frac{\nabla u\cdot \nu(s) }{Q }\cdot \frac{\varphi''(s)}{1+\varphi'(s)^2 }\, ds
-\int_{\partial \Omega}\frac{\kappa (s) }{Q }\, ds. \label{eq:Gauss-v2}
\end{eqnarray}
Thanks to our smoothness assumptions on the boundary data we find a function
$$
\alpha\in C^1(\overline{\Omega}) :\qquad \alpha|_{\partial\Omega} =\frac{\varphi''}{1+{\varphi'}^2 }|_{\partial\Omega}
$$
and may proceed:
\begin{equation}\label{eq:3.12}
   \int_{\Omega} K Q \, dx =2 \pi \chi(\Omega) +  \int_{ \Omega} H(x)\alpha(x)\, dx +\int_{ \Omega}\frac{\nabla u}{Q}\cdot \nabla \alpha \, dx
-\int_{\partial \Omega}\frac{\kappa (s) }{Q }\, ds.
\end{equation}

Now, consider a sequence  $u_k\to u \mbox{\ in\ } L^1(\Omega)$ as described in our assumptions.
According to the proof of Theorem~\ref{thm:L1lowersemicontinuity} and since $u\in H^2(\Omega)$ we have after passing to a suitable subsequence
\begin{eqnarray*}
 && H_k \rightharpoonup H\mbox{\ in\ } L^2(\Omega),\\
&& \nabla u_k \to \nabla u, \quad Q_k\to Q \mbox{\ a.e.~in \ } \Omega,\\
&& Q_k^{-5/2}=\left(1+|\nabla u_k|^2 \right)^{-5/4} \rightharpoonup Q^{-5/2}=\left(1+|\nabla u|^2 \right)^{-5/4} \mbox{\ in\ } H^1 (\Omega), \mbox{\ hence}\\
&& Q_k^{-5/2}  =\left(1+|\nabla u_k|^2 \right)^{-5/4}\rightarrow Q^{-5/2}=\left(1+|\nabla u|^2 \right)^{-5/4} \mbox{\ in\ } L^2 (\partial\Omega), \mbox{\ hence}\\
&& Q_k^{-5/2}  =\left(1+|\nabla u_k|^2 \right)^{-5/4}\rightarrow Q^{-5/2}=\left(1+|\nabla u|^2 \right)^{-5/4}  \quad\Ha^1\mbox{-a.e.~on \ } \partial\Omega.
\end{eqnarray*}
One should observe that thanks to $u_k,u\in H^2(\Omega)$, 
we also have $\nabla u_k|_{\partial\Omega}, \nabla u|_{\partial\Omega}\in L^2(\partial\Omega)$ and in particular
that $|\nabla u_k| <\infty $, $|\nabla u| <\infty $ $\Ha^1$-a.e.~on $\partial\Omega$.
We conclude that 
$$
Q_k=\sqrt{1+|\nabla u_k|^2 }\to Q=\sqrt{1+|\nabla u |^2 }\quad\Ha^1\mbox{-a.e.~on \ } \partial\Omega.
$$
Making use of Lebesgue's theorem and observing that $|\frac{\nabla u_k}{Q_k} |\le 1$ and $\frac{1 }{Q_k }\le 1$, this yields
\begin{eqnarray*}
    \int_{\Omega} K_k Q_k \, dx &=&2 \pi \chi(\Omega) +  \int_{ \Omega} H_k(x)\alpha(x)\, dx +\int_{ \Omega}\frac{\nabla u_k}{Q_k}\cdot \nabla \alpha \, dx
-\int_{\partial \Omega}\frac{\kappa (s) }{Q_k }\, ds\\
&\to &2 \pi \chi(\Omega) +  \int_{ \Omega} H(x)\alpha(x)\, dx +\int_{ \Omega}\frac{\nabla u}{Q}\cdot \nabla \alpha \, dx
-\int_{\partial \Omega}\frac{\kappa (s) }{Q }\, ds\\
&=& \int_{\Omega} K\,  Q \, dx.
\end{eqnarray*}
Since the previous reasoning can be carried out for any subsequence we have also convergence
of the whole sequence.
\end{proof}

\section{Minimising a relaxed Willmore functional in $L^1 (\Omega)$}\label{sec:minimiser}

\subsection{Dirichlet boundary conditions}\label{sec:minimiser_Dirichlet}

In what follows we always assume $\Omega\subset\mathbb{R}^2$ to be a bounded $C^2$-smooth domain
and fix a boundary datum $\varphi \in C^2(\overline{\Omega})$.

To model Dirichlet boundary conditions \eqref{eq:dirichlet}, i.e.
$$
u=\varphi \quad\mbox{\ and\ }\quad \frac{\partial u}{\partial \nu}=\frac{\partial\varphi }{\partial \nu} \quad\mbox{\ on\ }\partial \Omega,
$$
we consider the set
$$
\MDir:=\{ u\in H^2 (\Omega):\  (u-\varphi) \in H^2_0(\Omega)\}.
$$
As mentioned before (see Remark~\ref{rem:Gauss_Bonnet}) in this situation it suffices to consider the original Willmore functional $W_0$
since the total Gau\ss\ curvature is completely determined by the data.

Following Ambrosio \& Masnou \cite[Introduction \& Section 4]{AmbrosioMasnou} (cf. also \cite{BellettiniDalMasoPaolini,BellettiniMugnai}
and references therein), we  define the $L^1(\Omega)$--lower semicontinuous relaxation of the Willmore functional:
$$
\overline{W}:L^1( \Omega)\to [0,\infty],\quad \overline{W}(u):=\inf \{ \liminf_{k\to\infty} W_0(u_k): \MDir\ni u_k \to u 
\mbox{\ in\ } L^1 (\Omega)\}.
$$
We remark that such approximating sequences always exist. However, their Willmore energy may not be bounded and $\infty$ will certainly be attained by 
$\overline{W}$ for some $u\in L^1(\Omega)$. From the area and diameter bound we however obtain that any $u\in L^1(\Omega)$ with 
$\overline{W}(u)<\infty$ belongs at least to $BV(\Omega)\cap L^\infty(\Omega)$. 

One should observe that the Dirichlet boundary conditions are not encoded in the domain of definition of $\overline{W}$
but implicitly included by restricting the class of approximating sequences
to functions that satisfy the boundary conditions in $H^2(\Omega)$. We will prove below that
$\overline{W}(u)<\infty$ implies attainment of the Dirichlet boundary conditions in an appropriate weak sense.

We show first that $\overline{W}$ and $W_0$ coinicide on $\MDir$.

\begin{theorem}\label{thm:coinidence_relaxation}
For  $u\in\MDir$ one has $\overline{W}(u)=W_0(u)$.
\end{theorem}

\begin{proof}
For all $u\in\MDir$ the inequality $W_0(u)\ge \overline{W}(u)$ is obvious by definition.
To prove the opposite inequality take any sequence $(u_k)_{k\in \mathbb{N}}\subset\MDir $
with $u_k\to u $ in $L^1 (\Omega)$ and $\liminf_{k\to\infty} W_0(u_k)<\infty$. By
Theorem~\ref{thm:L1lowersemicontinuity} and since $u\in H^2(\Omega)$ we deduce
\begin{align*}
	W_0(u) \,=\, \Wnv_0(u) \,\le\, \liminf_{k\to\infty} W_0(u_k).
\end{align*}
This yields $W_0(u)\le \overline{W}(u)$.
\end{proof}

In the following proposition we discuss the implications of finiteness of $ \overline{W}(u)$.

\begin{proposition}\label{prop:W_vs_overline_W}
Suppose that $ \overline{W}(u)<\infty$ for some $u\in L^1(\Omega)$. Then 
\begin{align}
	u\,\in\, BV(\Omega)\cap L^\infty(\Omega),\qquad \frac{\nabla^a u}{Q^a}\in H(\operatorname{div},\Omega)\quad\text{ and }\quad
	\Wnv_0(u) \leq \overline{W}(u) \label{eq:Prop2-1}
\end{align}
holds. Moreover, 
both the trace of $u$ and the absolutely continuous representative of $\nabla^a u$ are well-defined $\Ha^1$-almost everywhere on $\partial\Omega$ 
and satisfy the Dirichlet boundary conditions \eqref{eq:dirichlet} $\Ha^1$-almost everywhere on $\partial\Omega$.
\end{proposition}

\begin{proof}
Let  $u\in L^1(\Omega)$ satisfy $\overline{W}(u)<\infty$. Then there exists  
a  sequence $(u_k)_{k\in\mathbb{N} }\subset \MDir$
such that  $u_k\to u$ in $L^1(\Omega)$ and $\overline{W}(u)=\lim_{k\to\infty} W_0(u_k) $. 
From Theorem~\ref{thm:L1lowersemicontinuity} we deduce \eqref{eq:Prop2-1}.

It therefore remains to prove the attainment of the boundary data (\ref{eq:dirichlet}). Let us choose an open bounded set ${\Omega}_1\subset\R^2$ 
with smooth boundary such that $\Omega\subset\subset\Omega_1$ and let us extend $\varphi$ to $\varphi\in C^2(\overline{\Omega_1})$. 
We also extend $u_k$ by $\varphi|_{\Omega_1\setminus\Omega}$ and obtain a sequence $(u_k)_k$ in $H^2(\Omega_1)$ 
with uniformly bounded Willmore energy also with respect to the larger domain $\Omega_1$. 
Theorem~\ref{thm:L1lowersemicontinuity} and the properties \eqref{eq:qe-conv-qk}, \eqref{eq:nablau-sing}, \eqref{eq:3.7-new2}, \eqref{eq:3.8}, \eqref{eq:3.9} show that 
\begin{align}
	q_k &\to q, \quad v_k \to v &&\text{ in }\Omega_1\setminus E,	\label{eq:prev1}\\
	 q\nabla^s u  \,&=\, 0\qquad&&\text{ as Radon measures on }\Omega_1, \label{eq:nablau-sing-bdry}
	\\
	\label{eq:3.7-new2-bdry}
	\nabla^a u \,&=\, q^{-1}v \quad&&\text{ on } \{q>0\}\setminus E,\\
	\nabla u_k \,&\to\,  \nabla^a u,\quad 
	Q_k \,\to\,  \sqrt{1+|\nabla^a u |^2 }  =Q^a \quad  &&\text{ in }\{q>0\}\setminus E,	
	\label{eq:prev3.9}
\end{align}
where $E\subset\Omega_1$ has $\frac{3}{2}$-capacity zero and the approximately continuous representative of $\nabla^a u$ exists everywhere in $\{q>0\}\setminus E$.
We recall that $q>0$ a.e. in $\Omega_1$.

On $\partial\Omega$ we have, denoting by $\tau$ a unit tangent field on $\partial\Omega$
\begin{align*}
	Q_k\,&=\, \sqrt{1+|\nabla u_k|^2} \,=\, \sqrt{1+(\nu\cdot\nabla u_k)^2 + (\tau\cdot\nabla u_k)^2}\,
=\, \sqrt{1+(\frac{\partial \varphi}{\partial\nu})^2 + (\frac{\partial\varphi}{\partial\tau})^2}.
\end{align*}
We deduce from \eqref{eq:dirichlet} and \eqref{eq:prev1} that
\begin{align}
	q\,>\, 0 \quad\text{ on }\partial\Omega\setminus E, \label{eq:q>0-bdry}
\end{align}
in particular $\Ha^{1}$-a.e.~on $\partial\Omega$. By \cite[Lemma 3.76]{AmFP00} we have $|\nabla^s u|(E)=0$ and by \eqref{eq:nablau-sing-bdry}, 
\eqref{eq:q>0-bdry} this yields $|\nabla^s u|(\partial\Omega)=0$. 

By \eqref{eq:prev3.9}, \eqref{eq:q>0-bdry} we deduce for the approximately continuous representative $\nabla^a u$, which is well-defined $\Ha^1$-a.e.~on $\partial\Omega$, that
\begin{align}
	\nabla^a u\cdot\nu \,&=\, \lim_{k\to\infty} \nabla u_k\cdot\nu \,=\, \nabla\varphi\cdot\nu\qquad \Ha^1\text{-almost everywhere on }\partial\Omega. \label{eq:bdry1}
\end{align}
This proves the attainment of the second Dirichlet boundary datum.

From now on we will work in the original domain $\Omega$. We observe that $g_k:=Q_k^{-3/2}$ and $e_k:=u_k\, Q_k^{-3/2}$ satisfy
\begin{eqnarray*}
	\partial_i g_k & = & -\frac{3}{2}\sum^2_{j=1} (\partial_ju_k)\, (\partial_j\partial_iu_k)\, Q_k^{-7/2},\\
	\partial_i e_k & = & (\partial_iu_k) Q_k^{-3/2}-\frac{3}{2}\sum^2_{j=1}u_k (\partial_ju_k)\, (\partial_j\partial_iu_k)\, Q_k^{-7/2}.
\end{eqnarray*}
Using the diameter bound (\ref{eq:3.4}) and  (\ref{eq:3.5}) we infer that the sequences $(g_k)_{k \in \mathbb{N}}$
 and $(e_k )_{k \in \mathbb{N}}$ are bounded in $H^1(\Omega)$.
After passing to suitable subsequences and possibly enlarging the set $E$, we obtain in addition to \eqref{eq:prev1}--\eqref{eq:q>0-bdry} that 
\begin{align}
	g_k\,&\rightharpoonup\, g,\quad e_k \rightharpoonup  e \quad\text{ in }H^1 (\Omega), \label{eq:gk-ek-H1}\\
	g_k\,&\to\, g,\quad e_k \to  e\quad
	\text{ in }\Omega\setminus E,\quad\capa_{3/2}(E)=0, \label{eq:gk-ek}
\end{align}
and that $g=(Q^a)^{-{3}/{2}}$ and $e=ug$ hold almost everywhere in $\Omega$.

We next claim that $\Ha^1$-almost everywhere on $\partial\Omega$ the traces of $e,g,u$, which are
well-defined by \cite[Theorem 5.3.1]{EvGa92}, satisfy $e=u g$.
In fact, $\Ha^1$-almost all $x_0\in \partial\Omega$ are by \cite[Theorem 3.87]{AmFP00} Lebesgue boundary points,
\begin{align*}
	\lim_{r\downarrow 0}\fint_{B_r(x_0)\cap\Omega} |u(x)-u(x_0)|\,dx \,&=\, 0,\\
	\lim_{r\downarrow 0}\fint_{B_r(x_0)\cap\Omega} |g(x)-g(x_0)|\,dx \,&=\, 0,\qquad
	\lim_{r\downarrow 0}\fint_{B_r(x_0)\cap\Omega} |e(x)-e(x_0)|\,dx \,=\, 0.
\end{align*}
Using these properties we deduce
\begin{align*}
	&|e(x_0)-u(x_0)g(x_0)|\\
	=\,&\fint_{B_r(x_0)\cap\Omega} |e(x_0)-u(x_0)g(x_0)|\,dx \\
	\leq\, & \fint_{B_r(x_0)\cap\Omega} \Big(|e(x_0)-e(x)| +
	 |u(x_0)(g(x_0)-g(x))| +
	 |g(x)(u(x_0)-u(x))|\Big)\,dx \\
	&\qquad + \fint_{B_r(x_0)\cap\Omega} |e(x)-u(x)g(x)|\,dx\\
	\to\,& 0\quad (r\downarrow 0),
\end{align*}
since the last integral is zero and since $u,g$ are uniformly bounded. This shows
\begin{align}
	ug \,=\, e\qquad\Ha^1\text{-almost everywhere on }\partial\Omega. \label{eq:ug=e}
\end{align}
We further obtain from $g=q^{3/5}$ in $H^1(\Omega)$ that $g=q^{3/5}$ holds $\Ha^1$-almost everywhere on $\partial\Omega$ for the corresponding traces. 
Furthermore the sets of Lebesgue boundary points of $q$ and Lebesgue boundary points of $g$ on $\partial\Omega$ are the same, since
\begin{align*}
	\fint_{B_r(x_0)\cap\Omega} |q^{3/5}-q^{3/5}(x_0)|\,dx \,
	\leq\, \fint_{B_r(x_0)\cap\Omega} |q-q(x_0)|^{3/5}\,dx
	\,\leq\, \fint_{B_r(x_0)\cap\Omega} |q-q(x_0)|\,dx,\\
	\fint_{B_r(x_0)\cap\Omega} |q-q(x_0)|\,dx\,\leq\, \fint_{B_r(x_0)\cap\Omega} |(q^{3/5})^{5/3}-(q(x_0)^{3/5})^{5/3}|\,dx
	\,\leq\, \fint_{B_r(x_0)\cap\Omega} \frac{5}{3}|q^{3/5}-q(x_0)^{3/5}|\,dx,
\end{align*}
where we have used $|q|\leq 1$. In particular, this implies 
\begin{align}
	\{g>0\}\cap (\partial\Omega\setminus B)\,=\, \{q>0\}\cap (\partial\Omega\setminus B) \quad\text{ for some }B\subset\partial\Omega \text{ with }\Ha^1(B)=0, \label{eq:qg>0}
\end{align}
and by \eqref{eq:q>0-bdry} 
\begin{align}
	g\,>0\, \quad\Ha^1\text{-almost everywhere on }\partial\Omega. \label{eq:g>0}
\end{align} 

By \eqref{eq:gk-ek-H1} and since $u_k$ satisfies the first Dirichlet boundary condition we find  that in $L^2(\partial\Omega)$ and $\Ha^1$-almost everywhere on $\partial\Omega$
\begin{align*}
	e \,=\,\lim_{k\to\infty} e_k\,=\,\lim_{k\to\infty} \varphi g_k\,=\, \varphi g
\end{align*}
holds. This yields by \eqref{eq:ug=e}, \eqref{eq:g>0} that $u=\varphi$ is satisfied $\Ha^1$-almost everywhere on $\partial\Omega$.
Together with \eqref{eq:bdry1} this proves that the Dirichlet boundary data are attained.
\end{proof}
Since by construction the lower semicontinuous relaxation is lower semicontinuous and by the compactness property from Theorem \ref{thm:L1lowersemicontinuity} 
we obtain the existence of a minimiser for $\overline{W}$, which is even bounded and has
finite surface area.

\begin{theorem}\label{thm:existence_minimiser}
There exists a function $u\in BV( \Omega)\cap L^\infty (\Omega)$ such that
$$
\forall v\in L^1( \Omega):\quad \overline{W}(u)\le \overline{W}(v).
$$
\end{theorem}

\begin{proof}
 We consider
$$
\alpha :=\inf\{\overline{W}(v) : v\in L^1( \Omega)\}\,<\,\infty
$$
and a minimising sequence $(u_k)_{k\in\mathbb{N} }\subset L^1( \Omega)$, thus $\alpha =\lim_{k\to\infty} \overline{W}(u_k)$.
Thanks to the definition of $\overline{W}$ we may achieve that even $(u_k)_{k\in\mathbb{N} }\subset \MDir$.
According to Theorem~\ref{thm:coinidence_relaxation} we have $\overline{W}(u_k)=W_0(u_k)$, hence 
$\alpha =\lim_{k\to\infty} {W_0}(u_k)$. 
Theorem \ref{thm:L1lowersemicontinuity} yields that for a subsequence $u_k\to u$ in $L^1(\Omega)$ and that $u\in BV( \Omega)\cap L^\infty (\Omega)$.
Due to the definition of $\overline{W}$ it follows that
$$
\overline{W}(u)\le \lim_{k\to\infty} {W_0}(u_k)=\alpha.
$$
The reverse inequality $\alpha \le \overline{W}(u)$ follows from the definition of $\alpha $ as an infimum.
To conclude we have $u\in BV (\Omega)\cap L^\infty (\Omega)$ and it satisfies
$$
\overline{W}(u)=\alpha =\inf\{\overline{W}(v) : v\in BV( \Omega)\}.
$$
\end{proof}

The preceding arguments show that the infimum in the definition of $\overline{W}$ is in fact a minimum and that
$$
\inf\{\overline{W}(v) : v\in L^1( \Omega)\}=\inf\{\overline{W}(v) : v\in \MDir\}=\inf\{ {W_0}(v) : v\in \MDir\}.
$$

\subsection{Navier boundary conditions}\label{sec:minimiser_Navier}

In what follows we always assume   $\Omega\subset\mathbb{R}^2$ to be  a bounded $C^3$-smooth domain
and fix a boundary datum $\varphi \in C^3(\overline{\Omega})$.

In order to model  the so called Navier boundary conditions
$$
u=\varphi \quad\mbox{\ and\ }\quad  H  =2\gamma \kappa_N \quad\mbox{\ on\ }\partial \Omega
$$
we consider the set
\begin{align}
	\MNav\,:=\,\{v\in H^2(\Omega):v-\varphi \in H^1_0(\Omega)\}. \label{eq:def-MNav}
\end{align}
As explained in the introduction one can formulate only the first Navier datum via a suitable subset of $H^2 (\Omega)$ while the 
second datum is only obtained via a minimising property and,
when compared with the Dirichlet setting, the larger set of admissible testing functions.

In contrast to Dirichlet boundary conditions the total Gau\ss\ curvature is not determined just by the Navier condition and is not constant on $\MNav$. 
Thus, we now consider the generalised Willmore functional $W_\gamma$ from \eqref{eq:def-Wgamma}. 

We  define as above the $L^1(\Omega)$-lower semicontinuous relaxation of the Willmore functional:
$$
\widehat{W}_\gamma :L^1(\Omega)\to [0,\infty],\quad \widehat{W}_\gamma(u):=\inf \{ \liminf_{k\to\infty} W_\gamma(u_k):  \MNav\ni u_k \to u 
\mbox{\ in\ } L^1 (\Omega)\}.
$$
We remark that again such approximating sequences always exist and that the set $\{\widehat{W}_\gamma <\infty\}$ will be strictly smaller than $L^1(\Omega)$.

Similarly to Theorem~\ref{thm:coinidence_relaxation} we also obtain for the Navier boundary problem that the relaxation of $W_\gamma$ 
coincides with the original functional in $\MNav$:

\begin{theorem}\label{thm:coinidence_relaxation_Navier}
 For  $u\in \MNav$ one has $\widehat{W}_\gamma(u)=W_\gamma(u)$.
\end{theorem}
\begin{proof}
The inequality $W_\gamma(u)\ge \widehat{W}_\gamma(u)$ follows immediately from the definition.
To prove the opposite inequality take any sequence $(u_k)_{k\in \mathbb{N}}\subset\MNav $
with $u_k\to u $ in $L^1 (\Omega)$ and $\liminf_{k\to\infty} W_\gamma(u_k)<\infty$. By Lemma~\ref{lemma:Abounds} also $(W_0 (u_k) )_{k\in\mathbb{N}}$ 
is bounded. Therefore all properties shown in
Theorem~\ref{thm:L1lowersemicontinuity} hold and since $u\in H^2(\Omega)$ we deduce that  
$
	W_0(u) \,\le\, \liminf_{k\to\infty} W_0(u_k).
$
Since the total Gau\ss\ curvature is continuous by Proposition~\ref{prop:L1lowersemicontinuity} we therefore also obtain \begin{align*}
	W_\gamma(u) \,\le\, \liminf_{k\to\infty} W_\gamma(u_k),
\end{align*}
which implies $W_\gamma(u)\le \widehat{W}_\gamma(u)$.
\end{proof}

\begin{remark}\label{rem:W_gamma_vs_widehat_W}
As in the Dirichlet case we would like to characterise properties of the subset of $L^1(\Omega)$ where $\widehat{W}_\gamma$ is finite. 
The key difficulty here is to identify a suitable generalisation of the total Gau\ss\ curvature for a sufficiently large subclass of functions 
$u\in L^1(\Omega)\setminus H^2(\Omega)$. We consider here for $u\in BV(\Omega)$ with $\frac{\nabla^a u}{Q^a}\in H(\dive,\Omega)$ and $(Q^a)^{-1}\in BV(\Omega)$ 
\begin{align}\label{eq:def-Gauss-new}
   \Gauss(u) \,:=\, 2 \pi \chi(\Omega) + \int_{\partial \Omega}\frac{\nabla^a u}{Q^a}\cdot \nu \, \frac{\partial_\tau^2\varphi}{1+(\partial_\tau\varphi)^2 }\, ds
-\int_{\partial \Omega}\frac{\kappa }{Q^a}\, ds,
\end{align}
where $\tau,\kappa$ denote a unit tangent field and the scalar curvature (taken nonnegative for convex parts) of $\partial\Omega$, respectively, see Remark~\ref{rem:geod_curv}.

Since $\frac{\nabla^a u}{Q^a}\in H(\dive,\Omega)$ by \cite[Theorem I.1.2]{Temam} we have $\frac{\nabla^a u}{Q^a}\cdot\nu\in H^{-1/2}(\partial\Omega)$ 
and the first boundary integral, which more precisely has to be understood as a $H^{-1/2}(\partial\Omega)$-$H^{1/2}(\partial\Omega)$ duality product, is well-defined. Furthermore $(Q^a)^{-1}\in BV(\Omega)$ ensures by \cite[Theorem 5.3.1]{EvGa92} that the second boundary integral is well-defined.
Note also that by \cite[Theorem I.1.2]{Temam} 
\begin{align}\label{eq:def-Gauss}
	\Gauss(u) \,=\, 2 \pi \chi(\Omega) +  \int_{ \Omega} H^a(x)\alpha(x)\, dx +\int_{ \Omega}\frac{\nabla^a u}{Q^a}\cdot \nabla \alpha \, dx
-\int_{\partial \Omega}\frac{\kappa }{Q^a}\, ds,
\end{align}
where $H^a:=\nabla\cdot \frac{\nabla^a u}{Q^a}$ and where $\alpha\in C^1(\overline{\Omega})$ is any differentiable function satisfying 
$\alpha|_{\partial\Omega}=\frac{\partial_\tau^2\varphi}{1+(\partial_\tau\varphi)^2 }|_{\partial\Omega}$.

Our choice of the functional $\Gauss$ is motivated by Proposition~\ref{prop:L1lowersemicontinuity} and \eqref{eq:Gauss-v2}. In fact, the latter proposition
 shows that $\Gauss(u)$ coincides with $\int_\Omega KQ\,dx$ for $u\in H^2(\Omega)$ with boundary values $\varphi$. Moreover, for all $C>0$ the functional $\Gauss$ is 
continuous with respect to $L^1$-convergence in $H^2(\Omega)\cap \{W_0\leq C\}$. We do not claim that our choice of $\Gauss$ is a reasonable representation for 
arbitrary $u\in L^1(\Omega)\setminus H^2(\Omega)$. However, at least for minimising sequences in $\MNav$ we expect that limit points enjoy additional good 
properties such that their total Gau\ss\ curvature might already be described by $\Gauss$.

Using this modified total Gau\ss\ curvature we can define a generalised Willmore functional (or rather the non-singular part of the latter) as
\begin{eqnarray*}
 \Wnv_\gamma (u) &=& \Wnv_0 (u)+\gamma \Gauss(u)
\end{eqnarray*}
for $u\in BV(\Omega)$ with $\frac{\nabla^a u}{Q^a}\in H(\dive,\Omega)$ and $(Q^a)^{-1}\in BV(\Omega)$.
\end{remark}
The next proposition shows that for $u\in L^1(\Omega)$ with $\widehat{W}_\gamma(u)<\infty$ the functional $\Wnv_\gamma$ is well-defined and satisfies 
an upper estimate. In addition, at the non-vertical part of the boundary the first Navier boundary datum is attained.

\begin{proposition}\label{prop:Wgamma}
Consider $u\in L^1(\Omega)$ with $\widehat{W}_\gamma(u)<\infty$. Then $u\in BV(\Omega)\cap L^\infty(\Omega)$, $\frac{\nabla^a u}{Q^a}\in H(\dive,\Omega)$ and 
$(Q^a)^{-1}\in BV(\Omega)$ holds and we have
\begin{align}
	\Wnv_\gamma(u) \,\leq\, \widehat{W}_\gamma(u). \label{eq:liminf-Wgamma}
\end{align}
Moreover, $\Ha^{1}$-almost everywhere on $\{(Q^a)^{-1}>0\}\cap\partial\Omega$ the trace of $u$ on $\partial\Omega$ satisfies the first Navier boundary condition $u=\varphi$.
\end{proposition}
\begin{proof}
There exists a sequence $(u_k)_{k\in\mathbb{N} }\subset \MNav$
such that
\begin{align*}
	u_k\to u\quad\text{ in }L^1(\Omega),\qquad
	\widehat{W}_\gamma(u)=\lim_{k\to\infty} W_\gamma(u_k).
\end{align*} 
Thanks to Lemma~\ref{lemma:Abounds} and since $(W_\gamma (u_k) )_{k\in\mathbb{N}}$ is bounded also $(W_0 (u_k) )_{k\in\mathbb{N}}$ is bounded.
So, most arguments of the proofs of Theorem \ref{thm:L1lowersemicontinuity} and Proposition~\ref{prop:W_vs_overline_W} carry over, but the convergence of 
$\int_\Omega K_k Q_k\, dx$ and the attainment of the boundary condition need to be carefully discussed. 

As in the proofs of Theorem \ref{thm:L1lowersemicontinuity} and Proposition~\ref{prop:W_vs_overline_W} we obtain, after passing to a subsequence and recalling 
$g_k=Q_k^{-3/2}$, $e_k= u_k g_k$,
\begin{align}
	g_k\,&\rightharpoonup\, g,\quad e_k \rightharpoonup  e \quad&&\text{ in }H^1 (\Omega),\notag\\
	g_k\,&\to\, g,\quad e_k \to  e\quad
	&&\text{ in }\Omega\setminus E,\quad\capa_{3/2}(E)=0, \label{eq:conv-ge}\\
	\nabla u_k \,&\to\, \nabla^a u \quad&&\text{ a.e.~in }\Omega, \label{eq:conv-nablau-g}\\
	H_k \,&\rightharpoonup\, H^a=\nabla\cdot \frac{\nabla^au}{Q^a} \quad&&\text{ in }L^2(\Omega), \label{eq:conv-H-g}
\end{align}
and that $g=(Q^a)^{-{3}/{2}}$ and $e=ug$ holds almost everywhere in $\Omega$.
Moreover $Q^a\in L^1(\Omega)$, $\frac{\nabla^a u}{Q^a}\in H(\dive,\Omega)$ and
\begin{align}
	\Wnv_0(u) \,\le\, \liminf_{k\to\infty}W_0(u_k). \label{eq:lsc-willmore-g}
\end{align}
Since $0<(Q_k)^{-1}\leq 1$ for all $k\in\N$ and further
\begin{align*}
	\nabla (Q_k)^{-1} \,=\, -Q_k^{-3}D^2u_k\nabla u_k,\qquad \int_\Omega |\nabla (Q_k)^{-1}| \,\leq\, 
\Big(\int_\Omega Q_k^{-5}|D^2u_k|^2\Big)^{\frac{1}{2}}\Big(\int_\Omega Q_k\Big)^{\frac{1}{2}},
\end{align*}
$(Q_k)^{-1}$ is uniformly bounded in $W^{1,1}(\Omega)$. By the BV compactness theorem and since $(Q_k)^{-1}\to (Q^a)^{-1}$ in $L^1(\Omega)$ we deduce that 
$(Q^a)^{-1}\in BV(\Omega)$.

We next show the convergence of the total Gau\ss\ curvature. Here it is convenient to fix any $\alpha\in C^1(\overline\Omega)$ as above and to use the 
representation \eqref{eq:def-Gauss}. By \eqref{eq:conv-ge} we deduce that $g_k\to g$ in $L^2(\partial\Omega)$ and, possibly passing to a subsequence, 
$\Ha^1$-almost everywhere on $\partial\Omega$. Since $g_k,g$ are bounded we deduce that we also have 
\begin{align*}
	(Q_k)^{-1}\,=\, g_k^{{2}/{3}}\,\to\, g^{{2}/{3}}\,=\, (Q^a)^{-1}
\end{align*}
strongly in $L^1(\partial\Omega)$, where in the last equality we have used that $g^{{2}/{3}}\,=\, (Q^a)^{-1}$ in $BV(\Omega)$ and therefore in $L^1(\partial\Omega)$. 
Furthermore, from \eqref{eq:conv-nablau-g} and since $|\frac{\nabla u_k}{Q_k}|\leq 1$ we obtain 
\begin{align*}
	\frac{\nabla u_k}{Q_k} \,\to\, \frac{\nabla^a u}{Q^a}\quad\text{ in }L^1(\Omega).
\end{align*}
Equation (\ref{eq:3.12}), the convergence properties just derived, and \eqref{eq:conv-H-g} yield for $k\to\infty$
\begin{align*}
	\Gauss(u_k)\,&=\, 2 \pi \chi(\Omega) +  \int_{ \Omega} H_k(x)\alpha(x)\, dx +\int_{ \Omega}\frac{\nabla u_k}{Q_k}\cdot \nabla \alpha \, dx
-\int_{\partial \Omega}\frac{\kappa }{Q_k}\, ds\\
 \,&\to\, 2 \pi \chi(\Omega) +  \int_{ \Omega} H^a(x)\alpha(x)\, dx +\int_{ \Omega}\frac{\nabla^a u}{Q^a}\cdot \nabla \alpha \, dx
-\int_{\partial \Omega}\frac{\kappa}{Q^a}\, ds \,=\, \Gauss(u). \label{eq:conv-gauss}
\end{align*}
Recalling \eqref{eq:lsc-willmore-g} we conclude that \eqref{eq:liminf-Wgamma} holds.

Following the proof of \eqref{eq:ug=e} in Proposition~\ref{prop:W_vs_overline_W} we obtain $\Ha^1$-almost everywhere $e=u g$ on $\partial\Omega$.
Moreover we deduce from $(Q^a)^{-1}=g^{2/3}$ similarly as in \eqref{eq:qg>0} that
\begin{align*}
	\{g>0\}\cap (\partial\Omega\setminus B)\,=\, \{(Q^a)^{-1}>0\}\cap (\partial\Omega\setminus B) \quad\text{ for some }B\subset\partial\Omega \text{ with }\Ha^1(B)=0, 
\end{align*}
and further by the first Navier boundary condition that $e =\lim_{k\to\infty} e_k= \varphi g$ holds $\Ha^1$-almost everywhere on $\partial\Omega$. 
This implies that $u=\varphi$ is satisfied $\Ha^1$-almost everywhere on the set $\partial\Omega\cap \{(Q^a)^{-1}>0\}$.
\end{proof}

\begin{remark}\label{rem:Wgamma}
As before we obtain as a corrollary the existence of a minimiser for $\widehat{W}$, which is even bounded and has
finite surface area:
There exists a function $u\in BV(\Omega)\cap L^\infty (\Omega)$ such that
$$
\forall v\in L^1(\Omega):\quad \widehat{W}_\gamma(u)\le \widehat{W}_\gamma(v).
$$
The proof follows closely that of Theorem~\ref{thm:existence_minimiser}. To obtain the respective compactness property for the generalised Willmore functional 
$W_\gamma$ we in addition use that by Lemma~\ref{lemma:Abounds} a bound on $W_\gamma$ implies a bound for $W_0$.
We expect that the first Navier boundary data are not necessarily attained in a pointwise sense if vertical parts of the graph are present in the limit. Such a deviation will 
be charged by contributions to the energy from the singular part. In particular in such cases we expect that $W_\gamma^a(u) < \widehat{W}_\gamma(u)$. 
\end{remark}

\begin{remark}\label{rem:Helfrich_parameters}
Most of the results for the functional $W_\gamma$ also apply to more general Canham--Helfrich-type functionals \cite{Helfrich,Canh70} 
$$
W_{\alpha, H_0,\gamma}(u) =\alpha \int_{\Omega} \sqrt{1+ | \nabla u |^2} \, dx +  \frac{1}{4} \int_{\Omega} (H-H_0)^2 \; \sqrt{1+ | \nabla u |^2} \, dx 
- \gamma \int_{\Omega} K\; \sqrt{1+ | \nabla u |^2}\, dx.
$$
The physical meaningful range of parameter values is described by the conditions $\alpha\ge 0$, $0\le \gamma\le 1$, $\gamma  H_0^2\le 4\alpha(1-\gamma)$, 
see \cite{Helfrich,Nitsche}. These restrictions ensure pointwise nonnegativity of the whole integrand  $\alpha +\frac{1}{4}(H-H_0)^2-\gamma K$.

Here we can consider arbitrary fixed $\alpha>0$ and $H_0, \gamma$. For given $\varphi\in C^2(\overline{\Omega})$ we prescribe the boundary 
condition $u|_{\partial\Omega}=\varphi|_{\partial\Omega}$. Then the term $\gamma \int_{\Omega} K\; \sqrt{1+ | \nabla u |^2}\, dx$ is 
uniformly bounded by the data, see the proof of Lemma~\ref{lemma:Abounds}. Hence, bounds for $W_{\alpha, \gamma, H_0}$
 immediately yield bounds for the area and so for $W_0$. Diameter bounds follow directly by Theorem~\ref{thm:diam}.

In order to extend Proposition~\ref{prop:Wgamma} one observes that the area term is $L^1$-lower semicontinuous.
Moreover, the proof of Theorem~\ref{thm:L1lowersemicontinuity} yields that $(\sqrt{Q_k})_{k\in\mathbb{N}}$ is bounded 
in $L^2(\Omega)$ and $\sqrt{Q_k} \to \sqrt{Q^a}$ holds almost everywhere in $\Omega$. Vitali's theorem implies that $\sqrt{Q_k}\rightharpoonup \sqrt{Q^a}$
in $L^2(\Omega)$. We conclude further from \eqref{eq:3.10} that $(H_k-H_0)\sqrt{Q_k}\rightharpoonup (H^a-H_0)\sqrt{Q^a} $ in $L^2(\Omega)$. 
Hence the proof of Proposition~\ref{prop:Wgamma} can be extended to the Helfrich case. 

If we assume only $\alpha\ge 0$, but further that $\alpha\ge \varepsilon H_0^2$
for some $\varepsilon >0$, then bounds for $W_{\alpha, \gamma, H_0}$
imply bounds for $W_0$ that are uniform in $\alpha$. Diameter and area bounds follow by Theorem~\ref{thm:aprioribounds}. The corresponding results to 
Proposition~\ref{prop:Wgamma} can the be proved as indicated above. 
\end{remark}

%
%
\bigskip\noindent
{\bf Acknowledgement.} The authors are grateful to Gerhard Huisken for very  fruitful and helpful discussions and
suggestions.


\end{document}